\documentclass[a4paper]{amsart}
\usepackage{graphicx}
\usepackage{amscd}
\usepackage{amsmath}
\usepackage{amsfonts}
\usepackage{amssymb}
\usepackage{setspace}
\usepackage{enumerate}         
\usepackage{fixme}
\usepackage{color}
\usepackage{url}
\usepackage{amsthm}
\usepackage{bm}
\usepackage{xy}
\usepackage{enumitem}
\usepackage{tikz-cd}
\usetikzlibrary{cd}
\usepackage{xcolor}

\theoremstyle{plain}
\newtheorem{theorem}{Theorem}[section]

\newtheorem{corollary}[theorem]{Corollary}

\newtheorem{lemma}[theorem]{Lemma}

\newtheorem{proposition}[theorem]{Proposition}

\newtheorem{definition}[theorem]{Definition}

\newtheorem{assumption}[theorem]{Assumption}
\theoremstyle{remark}
\newtheorem{remark}[theorem]{Remark}

\newtheorem{example}[theorem]{Example}

\numberwithin{equation}{section}

\newcommand{\ind}{1\!\kern-1pt \mathrm{I}}
\newcommand{\rsto}{]\!\kern-1.8pt ]}
\newcommand{\lsto}{[\!\kern-1.7pt [}

\renewcommand{\theenumi}{\rm{(\roman{enumi})}}
\numberwithin{equation}{section}

\newcommand{\diag}{\operatorname{diag}}

\renewcommand{\rho}{\varrho}

\DeclareMathOperator{\Tr}{Tr}

\begin{document}
\title[Markovian lifts of positive semidefinite affine Volterra
processes]{
  Markovian lifts of positive semidefinite affine Volterra type
  processes}

\begin{abstract}
  We consider stochastic partial differential equations appearing as
  Markovian lifts of matrix valued (affine) Volterra type processes
  from the point of view of the generalized Feller property (see e.g.,
  \cite{doetei:10}). We introduce in particular Volterra Wishart
  processes with fractional kernels and values in the cone of positive
  semidefinite matrices. They are constructed from matrix products of
  infinite dimensional Ornstein Uhlenbeck processes whose state space
  are matrix valued measures. Parallel to that we also consider
  positive definite Volterra pure jump processes, giving rise to
  multivariate Hawkes type processes. We apply these affine covariance processes for multivariate (rough) volatility modeling and introduce a (rough) multivariate Volterra Heston type model. 
\end{abstract}

\thanks{The authors are grateful for the support of the ETH Foundation
  and Erwin Schr\"odinger Institut Wien. Christa Cuchiero gratefully
  acknowledges financial support by the Vienna Science and Technology
  Fund (WWTF) under grant MA16-021.}

\keywords{stochastic partial differential equations, affine processes,
  Wishart processes, Hawkes processes, stochastic Volterra processes,
  rough volatility models} \subjclass[2010]{60H15, 60J25}

\author{Christa Cuchiero and Josef Teichmann}

\address{Vienna University of Economics and Business, Welthandelsplatz
  1, A-1020 Vienna and ETH Z\"urich, R\"amistrasse 101, CH-8092
  Z\"urich}
\maketitle

\section{Introduction}\label{sec:intro}

It is the goal of this article to investigate the results of
\cite{cuctei:18} on infinite dimensional Markovian lifts of stochastic
Volterra processes in a multivariate setup: we are mainly interested
in the case where the stochastic Volterra processes take values in the
cone of positive semidefinite matrices $\mathbb{S}^d_+$. We shall
concentrate on the affine case due to its relevance for tractable
\emph{rough covariance modeling}, extending rough volatility
(see e.g.,~\cite{ALV:07, GatJaiRos:14, BFG:16}) to a setting of $d$ ``roughly
correlated'' assets. 

Viewing stochastic Volterra processes from an infinite dimensional perspective allows to dissolve a generic non-Markovanity of the at first sight naturally low dimensional volatility process. Indeed, this approach makes it actually possible to go beyond the univariate case considered so far and  treat the  problem of multivariate rough covariance models for more than one asset.
Moreover, the considered Markovian lifts allow to apply the full machinery of affine processes. We refer to the introduction of \cite{cuctei:18} for an overview of theoretical and practical advantages of Markovian lifts in the context of Volterra type processes.

Let us start now by explaining why the matrix valued positive definite case is
actually more involved than the scalar one in $\mathbb{R}_+$, where
for instance the Volterra Cox-Ingersoll-Ross process takes values (see
e.g., \cite{ER:16, AbiElE:18b, AY:14} where it appears as
variance process in a rough Heston model): consider a standard Wishart
process on $\mathbb{S}_+^d$, as defined in \cite{B:91, CFMT:11}, of
the form
\begin{align}\label{eq:Wishart}
  d X_t = (d-1) \operatorname{Id}_{d} dt + \sqrt{X_t} dW_t + dW_t^\top
  \sqrt{X_t}, \quad X_0 \in \mathbb{S}^d_+.
\end{align}

Here $ \sqrt{.} $ denotes the matrix square root,
$\operatorname{Id}_d$ the identity matrix and $ W $ a $d \times d$
matrix of Brownian motions. The (necessary) presence of the dimension
$ d $ in the drift is an obvious obstruction to infinite dimensional
versions of this equation, which could be projected to obtain Volterra
type equations by the variation of constants formula (see
\cite{cuctei:18} for such a projection on $\mathbb{R}_+$). In order to
circumvent this difficulty we present two approaches in this paper:
\begin{itemize}
\item We develop a theory of infinite dimensional affine Markovian
  lifts of pure jump positive semidefinite Volterra processes.
\item We develop a theory of squares of Gaussian processes in a
  general setting to construct infinite dimensional analogs of Wishart
  processes. Their finite dimensional projections, however, look
  different from naively conjectured Volterra Wishart processes
  following the role model of Volterra Cox-Ingersoll-Ross
  processes. They are also different in dimension one, as outlined
  below.
\end{itemize}

The jump part appears natural and comes without any further
probabilistic problem when constrained to finite variation jumps. Note
that in the (non-Volterra) case of affine processes on positive
semidefinite matrices, quadratic variation jumps are not possible
either (see \cite{M:12}).  With the generalized Feller approach from
\cite{doetei:10, cuctei:18} we obtain a new class of stochastic
Volterra processes taking values in $\mathbb{S}_+^d$ of the form
\begin{align}\label{eq:Volterrajump}
  V_t= h(t)+\int_0^t( K(t-s) V_s + V_s K(t-s) )ds  + \int_0^t K(t-s) dN_s + \int dN_s K(t-s),
\end{align}
where $h: \mathbb{R}_+ \to \mathbb{S}^d_+$ is some deterministic
function, $K$ a (potentially fractional) kernel in
$L^2(\mathbb{R}_+, \mathbb{S}_+^d)$ and $N$ a pure jump process of
finite variation with jump sizes in $\mathbb{S}^d_+$, whose
compensator is a linear function in $V$. This allows for instance to
define a multivariate Hawkes process $\widehat{N}$ (see
\cite{hawkes:1971} for the one-dimensional case) with values in
$\mathbb{N}_0^d$ given by the diagonal entries of $N$, i.e., $\diag(N)=\widehat{N}$ and the compensator
of $\widehat{N}_i$ is given by $\int_0^{\cdot}V_{s,ii} ds$ (see
Example \ref{ex:Hawkes}). By means of the affine transform formula for
the infinite dimensional lift of \eqref{eq:Volterrajump}, we are able
to derive an expression for the Laplace transform of $V_t$ which can
be computed by means of matrix Riccati Volterra equations.

The difficulty of the continuous part arises from geometric
constraints, which can however be circumvent by building squares of
unconstrained processes. Let us illustrate the idea in a finite
dimensional setting: Let $W$ be an $n \times d $ matrix of Brownian
motions and let $\nu$ be a matrix in $\mathbb{R}^{d \times dk}$
consisting of $k$ submatrixes $\nu_i \in \mathbb{R}^{d \times d}$,
$i=1, \ldots, k$, i.e., $\nu=(\nu_1, \ldots, \nu_k)$.

Define now a Gaussian process with values in
$\mathbb{R}^{n \times dk}$ by $\gamma:= W\nu$.  Then, by It\^o's product
formula the $\mathbb{R}^{dk \times dk}$ valued process
$\gamma_t^\top \gamma_t $ satisfies the following equation
\begin{align}\label{eq:gammafinite}
  d \gamma_t^\top \gamma_t = n \nu^\top \nu dt + \nu^\top dW_t^\top \gamma_t +
  \gamma_t^\top dW_t \nu.
\end{align}
Following Marie-France Bru \cite[Subsection 5.2]{B:91} and setting
$\lambda_t:=\gamma_t^\top \gamma_t $, this can however also be written
via a $ kd \times kd $ matrix of independent Brownian motions $B$ satisfying
\begin{align}\label{eq:newBM}
  \sqrt{\gamma_t^\top \gamma_t} d B_t \sqrt{\nu^\top \nu} = \gamma_t^\top dW_t \nu
\end{align}
in the more familiar form 
\begin{align}\label{eq:trueWish}
  d \lambda_t=n \nu^\top \nu dt + \sqrt{\nu^\top \nu} dB_t^\top
  \sqrt{\lambda_t} + \sqrt{\lambda_t} dB_t \sqrt{\nu^\top \nu}
  \, .
\end{align}
Our article is devoted to analyze the situation where the index
variable $ \nu $ gets continuous, which is the only possible form of
an infinite dimensional Wishart process.  We believe that
generalized Feller processes are the right arena to achieve this
purpose. In this article we choose measure spaces, but an analogous
analysis can be done in the setting of function spaces as for instance
the Hilbert space setting of \cite{F:01} (see \cite[Section
5.2]{cuctei:18}). In the measure-valued setting we proceed as follows:
let $\gamma$ be an infinite dimensional Ornstein-Uhlenbeck process
taking values in $\mathbb{R}^{n \times d}$-valued regular Borel
measures on $\mathbb{R}_+$. Then Volterra Wishart processes arise as
finite dimensional projections of $\gamma^{\top}(dx_1)\gamma(dx_2)$ on
$\mathbb{S}_+^d$ and can be written as
\begin{equation}\label{eq:VolterraWishart}
  \begin{split}
    V_t&= h(t)+ n \int_0^t  K(t-s) K(t-s) ds \\
    &\quad + \int_0^t K(t-s) dW^{\top}_s Y(t,s) ds+ \int_0^t
    Y(t,s)^{\top} dW_s K(t-s),
  \end{split}
\end{equation}
where $h$ and $K$ are as in \eqref{eq:Volterrajump}, $W$ an
$n \times d $ matrix of Brownian motions and
$Y(t,s)= \int_0^{\infty} e^{-x (t-s)}\gamma_s(dx)$.  As explained in
Remark \ref{rem:squareOU}, $V_t$ corresponds to the matrix square of a
Volterra Ornstein Uhlenbeck process $X_t$, obtained as finite
dimensional projection of $\gamma(dx)$. The Volterra Wishart process
\eqref{eq:VolterraWishart} can then also be written in terms of the
forward process of $X_t$, i.e.
$(\mathbb{E}[X_t|\mathcal{F}_s])_{s\leq t}$, namely
\begin{equation*}
  \begin{split}
    V_t&= h(t)+ n \int_0^t  K(t-s) K(t-s) ds \\
    &\quad + \int_0^t K(t-s) dW^{\top}_s \mathbb{E}[X_t|\mathcal{F}_s]
    ds+ \int_0^t \mathbb{E}[X_t^{\top}|\mathcal{F}_s] dW_s K(t-s).
  \end{split}
\end{equation*}
Note that this is not of standard Volterra form, as e.g.~in
\cite{AbiLarPul:17}, since $Y(t,s)$ or $\mathbb{E}[X_t|\mathcal{F}_s]$
respectively cannot be expressed as a function of $V_t$. By moving to
a Brownian field analogous to \eqref{eq:newBM} it could however be
expressed as a path functional of $(V_s)_{s \leq t}$. For $n=d=1$ it
also gives rise to a different equation than the Volterra CIR process.
We explain the connection between \eqref{eq:VolterraWishart} and
\eqref{eq:gammafinite}-\eqref{eq:trueWish} in detail in Section
\ref{sec:square-OU}. 

Note that by choosing $K$ to be a matrix of
fractional kernels the trajectories of \eqref{eq:VolterraWishart}
become rough, whence $V$ qualifies for rough covariance modeling with
potentially different roughness regimes for different assets and their
covariances. This is in accordance with econometric observations. In Section \ref{sec:cov} we show how such models can be defined: we introduce a (rough) multivariate Volterra Heston type model with jumps and show that it can again be cast in the affine framework. This is particularly relevant for pricing basket or spread options using the Fourier pricing approach.

The remainder of the article is organized as follows: in Section
\ref{sec:not} we introduce some notation and review certain functional
analytic concepts. In Section \ref{sec:genFeller} and
\ref{sec:approximationTheorems}, we recall and extend results on
generalized Feller processes as outlined in \cite{cuctei:18}. In
particular, Theorem \ref{th:invariantspace} provides a result on
invariant (sub)spaces for generalized Feller processes that is crucial
for the square construction as outlined above.  In Sections
\ref{sec:markovianlift_abstract} we apply the presented theory to
SPDEs which are lifts of matrix valued stochastic Volterra jump
processes of type \eqref{eq:Volterrajump}. Section \ref{sec:square-OU}
is devoted to present a theory of infinite dimensional Wishart
processes which in turn give rise to (rough) Volterra Wishart
processes. In Section \ref{sec:cov} we apply these processes for multivariate (rough) volatility modeling.

\subsection{Notation and some functional analytic
  notions}\label{sec:not}

For the background in functional analysis we refer to the excellent
textbook \cite{SchWol:99} as main reference and to the equally
excellent books \cite{engnag:00,paz:83} for the background in strongly
continuous semigroups.

We shall apply the following notations: let $ Y $ be a Banach space
and $ Y^* $ its dual space, i.e.~the space of linear continuous
functionals with the strong dual norm
\[ {\|\lambda\|}_{Y^*} = \sup_{\|y\| \leq 1} | \langle y , \lambda
  \rangle | \, ,
\]
where $ \langle y , \lambda \rangle := \lambda(y) $ denotes the
evaluation of the linear functional $ \lambda $ at the point
$ y \in Y $. Since in the case of equation~\eqref{eq:Volterrajump},
cones $ \mathcal{E} $ of $ Y^* $ will be our statespaces, we denote
the polar cones in pre-dual notation, i.e.
\[
  \mathcal{E}_* = \big \{ y \in Y \, | \; \langle y , \lambda \rangle
  \leq 0 \text{ for all } \lambda \in \mathcal{E} \big \} .
\]
We denote spaces of bounded linear operators from Banach spaces
$ Y_1 $ to $ Y_2 $ by $ L(Y_1,Y_2) $ with norm
\[ {\| A \|}_{L(Y_1,Y_2)} := \sup_{{\|y_1\|}_{Y_1}\leq 1} {\| Ay_1
    \|}_{Y_2} \, .
\]
If $Y_1=Y_2$ we only write $\|\cdot \|_{L(Y_1)}$. On $ Y^* $ we shall
usually consider beside the strong topology (induced by the strong
dual norm) the weak-$*$-topology, which is the weakest locally convex
topology making \emph{all} linear functionals
$ \langle y , \cdot \rangle $ on $ Y^* $ continuous.  Let us recall
the following facts:
\begin{itemize}
\item The weak-$*$-topology is metrizable if and only if $ Y $ is
  finite dimensional: this is due to Baire's category theorem since
  $Y^*$ can be written as a countable union of closed sets, whence at
  least one has to contain an open set, which in turn means that
  compact neighborhoods exist, i.e.~a strictly finite dimensional
  phenomenon.
\item Norm balls $ K_R $ of any radius $ R $ in $ Y^* $ are compact
  with respect to the weak-$*$-topology, which is the Banach-Alaoglu
  theorem.
\item These balls are metrizable if and only if $ Y $ is separable:
  this is true since $ Y $ can be isometrically embedded into
  $ C(K_1) $, where $ y \mapsto \langle y,\cdot \rangle $, for
  $ y \in Y $. Since $ Y $ is separable, its embedded image is
  separable, too, which means -- by looking at the algebra generated
  by $ Y $ in $ C(K_1) $ -- that $ C(K_1) $ is separable, which is the
  case if and only if $ K_1 $ is metrizable.
\end{itemize}
Even though some results are more general, in particular often only
compactness of $K_R$ is used, we shall always assume separability in
this article.

Finally, a family of linear operators $ {(P_t)}_{t \geq 0} $ on a Banach space
$ Y $ with $ P_t P_s = P_{t+s} $ for $ s,t \geq 0 $ and with
$ P_0 =I $ where $I$ denotes the identity is called strongly
continuous semigroup if $ \lim_{t \to 0} P_t y = y $ holds true for
every $ y \in Y $. We denote its generator usually by $ A $ which is
defined as $ \lim_{t \to 0} \frac{P_t y - y}{t} $ for all
$ y \in \operatorname{dom}(A) $, i.e.~the set of elements where the
limit exists. Notice that $ \operatorname{dom}(A) $ is left invariant
by the semigroup $ P $ and that its restriction on the domain equipped
with the operator norm
\[ {\| y \|}_{\operatorname{dom}(A)} := \sqrt{\|y\|^2 + \|Ay\|^2}
\]
is again a strongly continuous semigroup.

Moreover, as already used in the introduction, $\mathbb{S}^d$ denotes
the vector space of symmetric $d \times d$ matrices and
$\mathbb{S}^d_+$ the cone of positive semidefinite ones. Furthermore, we denote by $\diag(A)$ the  vector consisting  of  the diagonal elements of a matrix $A$.

\section{Generalized Feller semigroups and processes}\label{sec:genFeller}

In the context of Markovian lifts of stochastic Volterra processes
(signed) measure valued processes appear in a natural way.  The
generalized Feller framework is taylor-made for such processes, as it
allows to consider non-locally compact state spaces. This we need
explicitely in Section~\ref{sec:square-OU} for Ornstein-Uhlenbeck
processes whose state space are matrix-valued measures. Beyond that
jump processes with unbounded but finite activity can be easily
constructed in this setting, see Proposition
\ref{prop:jump_perturbation} and
Section~\ref{sec:markovianlift_abstract}. We shall first collect some
results from \cite{cuctei:18} and generalize accordingly for the
purposes of this article.

\subsection{Defintions and results}
First we introduce weighted spaces and state a central
Riesz-Markov-Kakutani representation result. The underlying space $X$
here is a completely regular Hausdorff topological space.
\begin{definition}
  A function $\rho\colon X\to(0,\infty)$ is called \emph{admissible
    weight function} if the sets
  $K_R:=\left\{ x\in X\colon \rho(x)\le R \right\}$ are compact and
  separable for all $R>0$.
\end{definition}
An admissible weight function $\rho$ is necessarily lower
semicontinuous and bounded from below by a positive constant. We call
the pair $X$ together with an admissible weight function $\rho$ a
\emph{weighted space}. A weighted space is $\sigma$-compact. In the
following remark we clarify the question of local compactness of
convex subsets $\mathcal{E} \subset X$ when $X$ is a locally convex
topological space and $\rho$ convex.

\begin{remark}\label{rem:local_compactness}
  Let $X$ be a separable locally convex topological space and
  $\mathcal{E}$ a convex subset.  Moreover, let $\rho$ be a
  \emph{convex} admissible weight function.  Then $ \rho $ is
  continuous on $ \mathcal{E} $ if and only if $\mathcal{E}$ is
  locally compact. Indeed if $ \rho $ is continuous on $\mathcal{E}$,
  then of course the topology on $ \mathcal{E} $ is locally compact
  since every point has a compact neighborhood of type
  $ \{ \rho \leq R \} $ for some $ R > 0 $. On the other hand if the
  topology on $ \mathcal{E} $ is locally compact, then for every point
  $ \lambda_0 \in \mathcal{E} $ there is a a convex, compact
  neighborhood $ V \subset \mathcal{E}$ such that
  $ \rho(\lambda)-\rho(\lambda_0) $ is bounded on $ V $ by a number
  $ k > 0 $, whence by convexity
  $ |\rho(s(\lambda-\lambda_0)+\lambda_0)-\rho(\lambda_0)| \leq s k $
  for $ \lambda - \lambda_0 \in s(V-\lambda_0) $ and $ s \in ]0,1]
  $. This in turn means that $ \rho $ is continuous at $ \lambda_0 $.
\end{remark}

From now on $\rho$ shall always denote an admissible weight
function. For completeness we start by putting definitions for general
Banach space valued functions, although in the sequel we shall only
deal with $\mathbb{R}$-valued functions: let $Z$ be a Banach space
with norm ${\lVert \cdot \rVert}_Z $. The vector space
\begin{equation}
  \mathrm{B}^\rho(X;Z):=\left\{ f\colon X\to Z \colon \sup_{x\in X}\rho(x)^{-1}{\lVert f(x)\rVert}_Z < \infty \right\}
\end{equation}
of $ Z $-valued functions $ f $ equipped with the norm
\begin{equation}
  \lVert f\rVert_{\rho}:=\sup_{x\in X}\rho(x)^{-1}{\lVert f(x) \rVert}_Z,
\end{equation}
is a Banach space itself. It is also clear that for $Z$-valued bounded
continuous functions the continuous embedding
$\mathrm{C}_b(X;Z)\subset\mathrm{B}^\rho(X;Z)$ holds true, where we
consider the supremum norm on bounded continuous functions,
i.e. $\sup_{x \in X}\|f(x)\|$.
\begin{definition}
  We define $\mathcal{B}^{\rho}(X;Z)$ as the closure of
  $\mathrm{C}_b(X;Z)$ in $\mathrm{B}^{\rho}(X;Z)$. The normed space
  $\mathcal{B}^{\rho}(X;Z)$ is a Banach space.
\end{definition}

If the range space $Z=\mathbb{R}$, which from now on will be the case,
we shall write $ \mathcal{B}^\rho(X) $ for
$\mathcal{B}^{\rho}(X; \mathbb{R})$ and analogously $B^{\rho}(X)$.

We consider elements of $ \mathcal{B}^\rho(X) $ as continuous
functions whose growth is controlled by $ \rho $. More precisely we
have by \cite[Theorem 2.7]{doetei:10} that $f\in\mathcal{B}^{\rho}(X)$
if and only if $f|_{K_R}\in \mathrm{C}(K_R)$ for all $R>0$ and
\begin{equation}
  \label{eq:Bdecay}
  \lim_{R\to\infty}\sup_{x\in X\setminus K_R}\rho(x)^{-1}\lVert f(x)\rVert=0 \, .
\end{equation}
Additionally, by \cite[Theorem 2.8]{doetei:10} it holds that for every
$f\in\mathcal{B}^\rho(X)$ with $\sup_{x\in X}f(x)>0$, there exists
$z\in X$ such that
\begin{equation}
  \rho(x)^{-1}f(x) \le \rho(z)^{-1}f(z) \quad\text{for all $x\in X$},
\end{equation}
which emphasizes the analogy with spaces of continuous functions
vanishing at $ \infty $ on locally compact spaces.

Let us now state the following crucial representation theorem of Riesz
type:
\begin{theorem}[Riesz representation for
  $\mathcal{B}^\rho(X)$]\label{theorem:rieszrepresentation}
  For every continuous linear functional
  $\ell\colon\mathcal{B}^\rho(X)\to\mathbb{R}$ there exists a finite
  signed Radon measure $\mu$ on $X$ such that
  \begin{equation}
    \ell(f)=\int_{X}f(x)\mu(dx)\quad\text{for all $f\in\mathcal{B}^\rho(X)$.}
  \end{equation}
  Additionally
  \begin{equation}
    \label{eq:rieszrepresentation-psiintbound}
    \int_{X}\rho(x)\lvert \mu\rvert( dx) = \lVert \ell\rVert_{L(\mathcal{B}^\rho(X),\mathbb{R})},
  \end{equation}
  where $\lvert\mu\rvert$ denotes the total variation measure of
  $\mu$.
\end{theorem}

We shall next consider strongly continuous semigroups on
$ \mathcal{B}^\rho(X) $ spaces and recover very similar structures as
well known for Feller semigroups on the space of continuous functions
vanishing at $ \infty $ on locally compact spaces.

\begin{definition}\label{def:genFeller}
  A family of bounded linear operators
  $P_t\colon\mathcal{B}^{\rho}(X)\to\mathcal{B}^{\rho}(X)$ for
  $ t \geq 0 $ is called \emph{generalized Feller semigroup} if
  \begin{enumerate}
    \renewcommand{\theenumi}{{\bf F\arabic{enumi}}}
  \item
    \label{enu:defgenfeller-0id}
    $P_0=I$, the identity on $\mathcal{B}^{\rho}(X)$,
  \item $P_{t+s}=P_tP_s$ for all $t$, $s\ge 0$,
  \item
    \label{enu:defgenfeller-pwconv}
    for all $f\in\mathcal{B}^{\rho}(X)$ and $x\in X$,
    $\lim_{t\to 0}P_t f(x)=f(x)$,
  \item
    \label{enu:defgenfeller-bound}
    there exist a constant $C\in\mathbb{R}$ and $\varepsilon>0$ such
    that for all $t\in [0,\varepsilon]$,
    $\lVert P_t\rVert_{L(\mathcal{B}^{\rho}(X))}\le C $. 
  \item
    \label{enu:defgenfeller-positivity}
    $P_t$ is positive for all $t\ge 0$, that is, for
    $f\in\mathcal{B}^{\rho}(X)$, $f\ge 0$, we have $P_t f\ge 0$.
  \end{enumerate}
\end{definition}

We obtain due to the Riesz representation property the following key
theorem:
\begin{theorem}
  \label{theorem:Ttstrongcont}
  Let $(P_t)_{t\ge 0}$ satisfy (i) to (iv) of Definition
  \ref{def:genFeller}.  Then, $(P_t)_{t\ge 0}$ is strongly continuous
  on $\mathcal{B}^{\rho}(X)$, that is,
  \begin{equation}
    \lim_{t\to 0}\lVert P_t f-f\rVert_{\rho}=0
    \quad\text{for all $f\in\mathcal{B}^{\rho}(X)$}.
  \end{equation}
\end{theorem}
One can also establish a positive maximum principle in case that the
semigroup $ P_t $ grows around $ 0 $ like $ \exp(\omega t) $ for some
$\omega \in \mathbb{R}$ with respect to the operator norm on
$ \mathcal{B}^{\rho}(X) $. Indeed, the following theorem proved in
\cite[Theorem 3.3]{doetei:10} is a reformulation of the Lumer-Philips
theorem for pseudo-contraction semigroups using a \emph{generalized
  positive maximum principle} which is formulated in the sequel.
\begin{theorem}
  \label{theorem:Ttposmaxprinciple}
  Let $A$ be an operator on $\mathcal{B}^{\rho}(X)$ with domain $D$,
  and $\omega\in\mathbb{R}$.  $A$ is closable with its closure
  $\overline{A}$ generating a generalized Feller semigroup
  $(P_t)_{t\ge 0}$ with
  $\lVert P_t\rVert_{L(\mathcal{B}^{\rho}(X))}\le\exp(\omega t)$ for
  all $t\ge 0$ if and only if
  \begin{enumerate}
  \item $D$ is dense,
  \item $A-\omega_0$ has dense image for some $\omega_0>\omega$, and
  \item $A$ satisfies the generalized positive maximum principle, that
    is, for $f\in D$ with $(\rho^{-1}f)\vee 0\le \rho(z)^{-1}f(z)$ for
    some $z\in X$, $Af(z)\le \omega f(z)$.
  \end{enumerate}
\end{theorem}

As a new contribution to the general theorems we shall work out a
statement on invariant subspaces which will be crucial for
constructing squares of infinite dimensional OU-processes.

\begin{theorem}\label{th:invariantspace}
  Let $ X$ be a weighted space with weight $ \rho_1$, and
  $ q : X \to q(X) $ be a (surjective) continuous map from
  $ (X,\rho_1) $ to the weighted space $ (q(X),\rho_2) $. Let
  $ P^{(1)} $ be a generalized Feller semigroup acting on
  $ \mathcal{B}^{\rho_1}(X) $. Assume that
  $ \rho_2 \circ q \leq \rho_1 $ on $ X$. Let $D$ be a dense subspace
  of $\mathcal{B}^{\rho_2}(q(X)) $. Furthermore, for every
  $ f \in D \subset \mathcal{B}^{\rho_2}(q(X)) $ and for every
  $ t \geq 0 $, there is some $ g \in \mathcal{B}^{\rho_2}(q(X)) $
  such that
  \begin{align}\label{eq:assumptionP1}
    P^{(1)}_t (f \circ q) = g \circ q \, ,
  \end{align}
  and additionally there is a constant $ C \geq 1 $ such that
  \begin{align}\label{eq:assumptionP1_rho}
    P^{(1)}_t (\rho_2 \circ q) \leq C \rho_2 \circ q \, .
  \end{align}
  Then there is a generalized Feller semigroup $ P^{(2)} $ acting on
  $\mathcal{B}^{\rho_2}(q(X)) $ such that
  \begin{align}\label{eq:P2}
    P^{(1)}_t (f \circ q) = (P^{(2)}_t f) \circ q \, .
  \end{align}
\end{theorem}
\begin{proof}
  The continuous map $ q $ defines a linear operator $ M $ from
  $ \mathcal{B}^{\rho_2}(q(X)) $ to $ \mathcal{B}^{\rho_1}(X) $ via
  $ f \mapsto f \circ q $. Notice that $ M $ is bounded, since
  \[ {\| M f \|}_{\rho_1} \leq {\| f \|}_{\rho_2}, \quad f \in
    \mathcal{B}^{\rho_2}(q(X))
  \]
  due to the assumption $ \rho_2 \circ q \leq \rho_1 $. It is also
  injective, but its image is not necessarily closed.  Assumption
  \eqref{eq:assumptionP1} and \eqref{eq:assumptionP1_rho} now mean
  that
  \[
    P^{(1)}_t M f \in \operatorname{rg}(M)
  \]
  for every $ f \in \mathcal{B}^{\rho_2}(q(X)) $ and not only for
  $ f \in D $. Hence we can define
  \[
    P^{(2)}_t f := M^{-1} P^{(1)}_t M f \, ,
  \]
  which is by the very construction a semigroup of linear operators on
  $ \mathcal{B}^{\rho_2}(q(X)) $. Since $ M $ is continuous, its graph
  is closed, whence $ P^{(2)}_t $ is a bounded linear operator by the
  closed graph theorem. Moreover, property (iv) of Definition
  \ref{def:genFeller} holds true due to Assumption
  \eqref{eq:assumptionP1_rho}. Positivity is also preserved, since for
  $f \geq 0$ we have due to Assumption \eqref{eq:assumptionP1} and the
  fact that $P^{(1)}$ is a generalized Feller semigroup,
  \[
    P^{(2)}_t f = M^{-1} P^{(1)}_t M f=M^{-1}\underbrace{ P^{(1)}_t (f
      \circ q)}_{\geq 0}=M^{-1}(g \circ q)=g\geq 0.
  \]
  Here, $g$ is nonnegative due the positivity of
  $P^{(1)}_t (f \circ q)$.  By \eqref{eq:assumptionP1} and the
  definition of $P^{(2)}$, \eqref{eq:P2} clearly holds true. Hence,
  \[
    \lim_{t \to 0} P^{(2)}_t f (q(x)) = \lim_{t \to 0} P^{(1)}_t
    f(q(x)) = f(q(x))
  \]
  for $ x \in X$ and thus property (iii) of Defintion
  \ref{def:genFeller}. Hence all conditions of Definition
  \ref{def:genFeller} are satisfied and we can conclude that the
  operators $ (P_t^{(2)}) $ form a generalized Feller semigroup.
\end{proof}

\begin{remark}
  In the setting of general semigroups it is not clear that
  restrictions of semigroups to (not even closed) subspaces preserve
  strong continuity.
\end{remark}
\begin{remark}
  There are several methods to show that \eqref{eq:assumptionP1} is
  satisfied. In general it is not sufficient to assume that the
  generator of $ P^{(1)} $ has this property.
\end{remark}

\begin{corollary}\label{cor:closed_invariant_subspace}
  Let the assumptions of Theorem \ref{th:invariantspace} except
  Assumption \eqref{eq:assumptionP1_rho} hold true and suppose
  additionally that
  \[
    \rho_2 \circ q = \rho_1. \,
  \]
  Then the same conclusions hold true. In particular the range of the
  operator
  $ M: \mathcal{B}^{\rho_2}(q(X)) \to \mathcal{B}^{\rho_1}(X),\, f
  \mapsto f \circ q $ is closed.
\end{corollary}

We restate from \cite{cuctei:18} assertions on existence of generalized Feller processes and path
properties. It is remarkable that in this
very general context c\`ag versions exist for countably many
test functions.

\begin{theorem}\label{th:kolmogorov_extension}
  Let $ (P_t)_{t \geq 0} $ be a generalized Feller semigroup with
  $ P_t 1 = 1 $ for $ t \geq 0 $. Then there exists a filtered
  measurable space $ (\Omega,(\mathcal{F}_t)_{t \geq 0}) $ with right
  continuous filtration, and an adapted family of random variables
  $ {(\lambda_t)}_{t \geq 0} $ such that for any initial value
  $ \lambda_0 \in X $ there exists a probability measure
  $ \mathbb{P}^{\lambda_0} $ with
  \[
    \mathbb{E}_{\lambda_0}[f(\lambda_t)] :=
    \mathbb{E}_{\mathbb{P}^{\lambda_0}}[f(\lambda_t)] = P_t
    f(\lambda_0)
  \]
  for $ t \geq 0 $ and every $ f \in \mathcal{B}^\rho(X) $. The Markov
  property holds true, i.e.
  \[
    \mathbb{E}_{\mathbb{P}^{\lambda_0}}[f(\lambda_t) \, | \;
    \mathcal{F}_s] = P_{t-s} f(\lambda_s)
  \]
  almost surely with respect to $ \mathbb{P}^{\lambda_0} $.
\end{theorem}

\begin{theorem}\label{th:path_properties}
  Let $ (P_t)_{t\geq 0} $ be a generalized Feller semigroup and let
  $ (\lambda_t)_{t \geq 0} $ be a generalized Feller process on a
  filtered probability space. Then for every countable family
  $ {(f_n)}_{n \geq 0} $ of functions in $ \mathcal{B}^\rho(X) $ we
  can choose a version of the processes
  $ {\left( \frac{f_n(\lambda_t)}{\rho(\lambda_t)} \right)}_{t \geq 0}
  $, such that the trajectories are c\`agl\`ad for all $ n \geq 0
  $. 
  If additionally $ P_t \rho \leq \exp(\omega t) \rho $ holds true,
  then $ (\exp(- \omega t) \rho(\lambda_t))_{t \geq 0} $ is a
  super-martingale and can be chosen to have c\`agl\`ad
  trajectories. In this case we obtain that the processes
  $ {\big( f_n(\lambda_t) \big)}_{t \geq 0} $ can be chosen to have
  c\`agl\`ad trajectories.
\end{theorem}

\begin{remark}\label{rem:pathproperties}
 In the general case, when
    $ P_t \rho \leq M \exp(\omega t) \rho $ for $M >1$, we obtain for
    $ {\big( f_n(\lambda_t) \big)}_{t \geq 0} $ only c\`ag
    trajectories. To see this, consider the measurable set of sample
    events $ \{ \sup_{0 \leq t \leq 1} \rho(\lambda_t) \leq R \}
    $. Then we can construct on the metrizable compact set
    $ \{ \rho \leq R \} $ a c\`agl\`ad version of the processes
    $ {\left( \frac{f_n(\lambda_t)}{\rho(\lambda_t)} \right)}_{t \leq
      1} $ and $ \left({\frac{1}{\rho(\lambda_t)}}\right)_{t \leq 1} $
    and in turn also of $ {\big( f_n(\lambda_t) \big)}_{t \geq
      0}$. The limit $ R \to \infty $, however, only leads to a c\`ag
    version since we cannot control the right limits.
\end{remark}

\subsection{Dual spaces of Banach spaces} \label{subsec:dual}

The most important playground for our theory will be closed subsets of
duals of Banach spaces, where the weak-$*$-topology appears to be
$ \sigma $-compact due to the Banach-Alaoglu theorem. Assume that
$ \mathcal{E} \subset Y^*$ is a closed subset of the dual space $Y^*$
of some Banach space $Y$ where $Y^{\ast}$ is equipped with its
weak-$*$-topology. Consider a lower semicontinuous function
$\rho\colon \mathcal{E} \to(0,\infty)$ and denote by
$(\mathcal{E},\rho)$ the corresponding weighted space. We have the
following approximation result (see \cite[Theorem 4.2]{doetei:10}) for
functions in $\mathcal{B}^{\rho}(\mathcal{E})$ by cylindrical
functions. Set
\begin{alignat}{2}
  \operatorname{Cyl}_N := \bigl\{
  g(\langle\cdot,y_1\rangle,\dots,\langle\cdot,y_N\rangle)\colon
  &\text{$g\in\mathrm{C}_b^{\infty}(\mathbb{R}^N)$} \notag
  \\
  &\text{and $y_j\in Y$, $j=1,\dots,N$} \bigr\},
\end{alignat}
where $\langle \cdot, \cdot \rangle$ denotes the pairing between $Y^*$
and $Y$. We denote by
$\operatorname{Cyl}:=\bigcup_{N\in\mathbb{N}}\operatorname{Cyl}_N$ the
set of bounded smooth continuous cylinder functions on $\mathcal{E}$.
\begin{theorem}
  \label{theorem:boundedweakcontapprox}
  The closure of $\operatorname{Cyl}$ in
  $\mathrm{B}^\rho(\mathcal{E})$ coincides with
  $\mathcal{B}^\rho(\mathcal{E})$, whose elements appear to be
  precisely the functions $f\in\mathcal{B}^{\rho}(\mathcal{E})$ which
  satisfy \eqref{eq:Bdecay} and that $f|_{K_R}$ is weak-$*$-continuous
  for any $R>0$.
\end{theorem}
\begin{proof} See \cite{cuctei:18}.
\end{proof}

\begin{assumption}\label{ass:generic}
  Let $(\lambda_t)_{t\ge 0}$ denote a time homogeneous Markov process
  on some stochastic basis
  $(\Omega,\mathcal{F}, (\mathcal{F}_t)_{t\ge 0},
  \mathbb{P}^{\lambda_0})$ with values in $\mathcal{E}$.
  
  Then we assume that
  \begin{enumerate}
  \item there are constants $C$ and $\varepsilon>0$ such that
    \begin{equation}
      \label{eq:markovpsibound}
      \mathbb{E}_{\lambda_0}[\rho(\lambda_t)]\le C\rho(\lambda_0)
      \quad\text{for all $\lambda_0\in \mathcal{E}$ and $t\in[0,\varepsilon]$};
    \end{equation}
  \item
    \begin{equation} \label{eq:contint} \lim_{t\to 0}
      \mathbb{E}_{\lambda_0}[f(\lambda_t))] = f(\lambda_0) \quad
      \text{for any $f\in\mathcal{B}^{\rho}(\mathcal{E})$ and
        $\lambda_0\in \mathcal{E}$};
    \end{equation}
  \item for all $f$ in a dense subset of
    $ \mathcal{B}^\rho(\mathcal{E}) $, the map
    $ \lambda_0 \mapsto \mathbb{E}_{\lambda_0}[f(\lambda_t)] $ lies in
    $ \mathcal{B}^\rho(\mathcal{E}) $.
  \end{enumerate}
\end{assumption}

\begin{remark}
  Of course inequality \eqref{eq:markovpsibound} implies that
  $ \lvert\mathbb{E}_{\lambda_0}[f(\lambda_t)]\rvert \leq C
  \rho(\lambda_0) $ for all $ f \in \mathcal{B}^{\rho}(\mathcal{E}) $,
  $ \lambda_0 \in \mathcal{E} $ and $ t \in [0,\varepsilon]$.
\end{remark}

\begin{theorem}
  \label{theorem:strongcontprocess}
  Suppose Assumptions \ref{ass:generic} hold true. Then
  $P_t f(\lambda_0):=\mathbb{E}_{\lambda_0}[f(\lambda_t)]$ satisfies
  the generalized Feller property and is therefore a strongly
  continuous semigroup on $\mathcal{B}^\rho(\mathcal{E})$.
\end{theorem}

\begin{proof}
  This follows from the arguments of \cite[Section 5]{doetei:10}.
\end{proof}

\section{Approximation theorems}\label{sec:approximationTheorems}

In order to establish existence of Markovian solutions for general
generators $A$ we could at least in the pseudo-contrative case either directly apply Theorem
\ref{theorem:Ttposmaxprinciple}, where we have to assume that the
generator $A$ satisfies on a dense domain $D$ a generalized positive
maximum principle and that for at least one $ \omega_0 > \omega $ the
range of $ A - \omega_0 $ is dense, or we approximate a general
generator $A$ by (finite activity pure jump) generators $A^n $ and
apply the following (well known) approximation theorems. They also work in the general context when the constant $M >1$.

\begin{theorem}\label{thm:approximation}
  Let $ (P_t^n)_{n \in \mathbb{N}, t\geq 0} $ be a sequence of
  strongly continuous semigroups on a Banach space $Z$ with generators
  $ (A^n)_{n \in \mathbb{N}} $ such that there are uniform (in $n$)
  growth bounds $ M \geq 1 $ and $ \omega \in \mathbb{R} $ with
  \begin{align}\label{eq:growthbounduni}
    \| P^n_t \|_{L(Z)} \leq M \exp(\omega t)
  \end{align}
  for $ t \geq 0 $. Let furthermore
  $ D \subset \cap_n \operatorname{dom}(A^n)$ be a dense subspace with
  the following three properties:
  \begin{enumerate}
  \item $D$ is an invariant subspace for all $ P^n $, i.e.~for all
    $ f \in D $ we have $ P^n_t f \in D $, for $ n \geq 0 $ and
    $ t \geq 0 $.
  \item There is a norm $ {\|.\|}_D $ on $ D $ such that there are
    uniform growth bounds with respect to $ {\|.\|}_D $, i.e.~there
    are $ M_D \geq 1 $ and $ \omega_D \in \mathbb{R} $ with
    \[ {\| P^n_t f \|}_D \leq M_D \exp(\omega_Dt) {\|f\|}_D
    \]
    for $ t \geq 0 $ and for $ n \geq 0 $.
  \item The sequence $ A^n f $ converges as $ n \to \infty $ for each
    $ f \in D $, in the following sense: there exists a sequence of
    numbers $ a_{nm} \to 0 $ as $ n,m \to \infty $ such that
    \[
      \| A^n f - A^m f \| \leq a_{nm} {\| f \|}_D
    \]
    holds true for every $ f \in D $ and for all $n,m$.
  \end{enumerate}
  Then there exists a strongly continuous semigroup
  $ (P_t^\infty)_{t \geq 0} $ with the same growth bound on $ Z $ such
  that $ \lim_{n \to \infty} P^n_t f = P^\infty_t f $ for all
  $ f \in Z $ uniformly on compacts in time and on bounded sets in
  $D$. Furthermore on $ D $ the convergence is of order $ O(a_{nm})
  $. If in addition for each $n \in \mathbb{N}$, $(P_t^n)_{t \geq 0}$
  is a generalized Feller semigroup, then this property transfers also
  to the limiting semigroup.
\end{theorem}
\begin{proof} See \cite{cuctei:18}.
\end{proof}

For the purposes of affine processes a slightly more general version
of the approximation theorem is needed, which we state in the sequel:

\begin{theorem}\label{thm:approximation_gen}
  Let $ (P_t^n)_{n \in \mathbb{N}, t\geq 0} $ be a sequence of
  strongly continuous semigroups on a Banach space $Z$ with generators
  $ (A^n)_{n \in \mathbb{N}} $ such that there are uniform (in $n$)
  growth bounds $ M \geq 1 $ and $ \omega \in \mathbb{R} $ with
  \[
    \| P^n_t \|_{L(Z)} \leq M \exp(\omega t)
  \]
  for $ t \geq 0 $. Let furthermore
  $ D \subset \cap_n \operatorname{dom}(A^n)$ be a \emph{subset} with
  the following two properties:
  \begin{enumerate}
  \item The linear span $\operatorname{span}(D)$ is dense.
  \item There is a norm $ {\|.\|}_D $ on $ \operatorname{span}(D) $
    such that for each $ f \in D $ and for $ t > 0 $ there exists a
    sequence $ a^{f,t}_{nm} $, possibly depending on $ f $ and $t$,
    \[
      \| A^n P^m_u f - A^m P^m_u f \| \leq a^{f,t}_{nm} {\| f \|}_D
    \]
    holds true for $ n, m $ and for $ 0 \leq u \leq t$, with
    $ a^{f,t}_{nm} \to 0 $ as $ n,m \to \infty $.
  \end{enumerate}
  Then there exists a strongly continuous semigroup
  $ (P_t^\infty)_{t \geq 0} $ with the same growth bound on $ Z $ such
  that $ \lim_{n \to \infty} P^n_t f = P^\infty_t f $ for all
  $ f \in Z $ uniformly on compacts in time. If in addition for each
  $n \in \mathbb{N}$, $(P_t^n)_{t \geq 0}$ is a generalized Feller
  semigroup, then this property transfers also to the limiting
  semigroup.
\end{theorem}
\begin{proof} See \cite{cuctei:18}.
\end{proof}

Our first application of Theorem \ref{thm:approximation} is the next
proposition that extends well-known results on bounded generators
towards unbounded limits.

We repeat here a remark from \cite{cuctei:18} since it helps to
understand the fourth condition on the measures:

\begin{remark}\label{rem:quasicontractive}
  Let $ (P_t)_{t \geq 0} $ be a generalized Feller semigroup with $\| P_t\|_{L(\mathcal{B}^{\rho}(X))}\leq M \exp(\omega t) $ for some $M\geq 1 $ and some $\omega$. Additionally it is assumed to be of transport type, i.e.
  \begin{align}\label{eq:transport}
    P_t f (x) = f ( \psi_t(x))
  \end{align}
  for some continuous map $ \psi_t : X \to X $. Define now a new
  function
  \[
    \tilde \rho (x) := \sup_{t \geq 0} \, \exp(-\omega t) P_t \rho(x)
  \]
  for $ x \in X $. Notice that $ \tilde \rho $ is an admissible weight
  function, since
  \[
    \{ \tilde \rho \leq R \} = \cap_{t \geq 0} \, \{ P_t \rho \leq
    \exp(\omega t) R \} \leq \{ \rho \leq R \}
  \]
  is compact by the definition of $\rho$ and the continuity of $x \mapsto \psi_t(x)$ which leads to an 
  intersection of closed subsets of compacts. Additionally we have
  that
  \[
    \rho \leq \tilde \rho \leq M \rho
  \]
  by the growth bound and therefore the norm on
  $ \mathcal{B}^\rho(X) $ is equivalent to
  \[ {\lVert f \rVert}_{\tilde \rho} = \sup_{x \in X}
    \frac{|f(x)|}{\tilde \rho(x)} \, .
  \]
  Furthermore,
  \[
    \lVert P_tf \rVert_{\tilde \rho} \leq \exp(\omega t) \lVert f
    \rVert_{\tilde \rho}
  \]
  holds for all $t\ge 0$ and $ f \in \mathcal{B}^\rho(X) $.  Indeed,
  this is a consequence of the following estimate
  \begin{align*}
    \lVert P_tf \rVert_{\tilde \rho} &= \sup_{x}\left|\frac{f(\psi_t(x))}{\sup_s \exp(-\omega s) \rho(\psi_s(x))}\right|\leq \sup_{x}\left|\frac{f(\psi_t(x))}{\sup_s \exp(-\omega (t+s)) \rho(\psi_{t+s}(x))}\right|    
    \\
                                     &\leq \exp(\omega t) \sup_{x}\left|\frac{f(\psi_t(x))}{\sup_s \exp(-\omega s)  \rho(\psi_{s}(\psi_t(x)))}\right|\leq  \exp(\omega t) \|f\|_{\tilde{\rho}}.
  \end{align*}
  Hence,
  \[
    |P_tf(x)| \leq \exp(\omega t)\tilde{\rho}(x) \|f\|_{\tilde{\rho}},
  \]
  which implies
  \[
    P_t \tilde \rho \leq \exp(\omega t) \tilde \rho, \quad t \geq 0.
  \]
\end{remark}

\begin{proposition}\label{prop:jump_perturbation}
  Let $ (X,\rho) $ be a weighted space with weight function
  $ \rho \geq 1 $. Consider an operator $A$ on $\mathcal{B}^{\rho}(X)$
  with dense domain $\operatorname{dom}(A)$ generating on
  $ \mathcal{B}^\rho(X) $ a generalized Feller semigroup
  $ (P_t)_{t\geq 0} $ of transport type as in \eqref{eq:transport},
  such that for all $t \geq 0$ we have
  $ \|P_t\|_{L(B^{\rho}(X))} \leq M_1 \exp(\omega t)$ for some $M_1$
  and $ \omega $ and such that
  $ \mathcal{B}^{\sqrt{\rho}}(X) \subset \mathcal{B}^\rho(X) $ is left
  invariant.

  Consider furthermore a family of finite measures $ \mu(x,.)$ for
  $ x \in X $ on $ X $ such that the operator $B$ acts on
  $\mathcal{B}^{\rho}(X)$ by
  \[
    B f (x) : = \int (f(y) - f(x)) \mu(x,dy)
  \]
  for $ x \in X $ yielding continuous functions on
  $ \{\rho \leq R \} $ for $ R \geq 0 $, and such that the following
  properties hold true:
  \begin{itemize}
  \item For all $ x \in X $
    \begin{align}\label{eq:cond1} \int \rho(y)
      \mu(x,dy) \leq M \rho^2 (x),
    \end{align}
    as well as
    \begin{align} \label{eq:cond2} \int \sqrt{\rho(y)} \mu(x,dy) \leq
      M \rho (x),
    \end{align}
    and
    \begin{align}\label{eq:cond3}
      \int \mu(x,dy) \leq M \sqrt{\rho (x)},
    \end{align}
    hold true for some constant $M$.
  \item For some constant $ \widetilde{\omega} \in \mathbb{R} $
    \begin{align}\label{eq:cond4}
      \int \Big | \frac{\sup_{t \geq 0} \exp(-\omega t) P_t \rho(y) -\sup_{t \geq 0} \exp(- \omega t ) P_t \rho(x)}{\sup_{t \geq 0} \exp(-\omega t) P_t \rho(x)} \Big | \mu(x,d y)
      \leq \widetilde{\omega} ,
    \end{align}
    for all $ x \in X $. In particular
    $ y \mapsto \sup_{t \geq 0} \exp(-\omega t) P_t \rho(y) $ should
    be integrable with respect to $ \mu(x,.) $
  \end{itemize}
  Then $ A + B $ generates a generalized Feller semigroup
  $(P_t^{\infty})_{t \geq 0}$ on $ \mathcal{B}^\rho(X) $ satisfying
  $\|P^{\infty}_t\|_{L(\mathcal{B}^{\rho}(X))} \leq M_1 \exp((\omega +
  \tilde{\omega})t)$.
\end{proposition}

\begin{proof} See \cite{cuctei:18}.
\end{proof}

\begin{remark}
  In contrast to classical Feller theory also processes with unbounded
  jump intensities can be constructed easily if $ \rho $ is unbounded
  on $ X $. The general character of the proposition allows to build
  general processes from simple ones by perturbation.
\end{remark}

\section{Lifting Stochastic Volterra jump processes with values in $\mathbb{S}^d_+ $} \label{sec:markovianlift_abstract}

Building on the theory of generalized Feller proceses from above, we
shall now treat the following type of matrix-measure valued SPDEs
\begin{equation}
  \begin{split} \label{eq:SPDE_weak}
    d \lambda_t(dx) &= \mathcal{A}^* \lambda_t(dx) dt + \nu(dx) dX_t + dX_t \nu(dx), \\
    \lambda_0 & \in \mathcal{E} .
  \end{split}
\end{equation}
As shown below this equation corresponds to a Markovian lift of the
Volterra jump process in \eqref{eq:Volterrajump}.

We consider here the setting of Section \ref{subsec:dual}.  The
underlying Banach space $Y^*$ is here the space of finite
$\mathbb{S}^d$-valued regular Borel measures on the extended half real
line $\overline{\mathbb{R}}_+:=\mathbb{R}_+ \cup \{\infty\}$ and
$\mathcal{E}$ denotes a (positive definite) subset of $Y^*$.
Moreover, $\mathcal{A}^*$ is the generator of a strongly continuous
semigroup $\mathcal{S}^*$ on $Y^*$, $\nu \in Y^*$ (or in a slightly
larger space denoted by $Z^*$ in the sequel).  The predual space $Y$
is given by $C_{b}(\overline{\mathbb{R}}_+, \mathbb{S}^d)$
functions. Note that since $\overline{\mathbb{R}}_+$ is compact,
$Y=C_{b}(\overline{\mathbb{R}}_+, \mathbb{S}^d)$ is separable.  The
driving process $X$ is an $\mathbb{S}^d$-valued pure jump
It\^o-semimartingale, whose differential characteristics depend
linearly on $\lambda$, precisely specified below.  Let us remark that
other forms of differential characteristics of $X$, in particular
beyond the linear case, can be easily incorporated in this setting.

The pairing between $Y$ and $Y^*$, denoted by
$\langle \cdot, \cdot \rangle $, is specified via:
\[
  \langle \cdot, \cdot \rangle: Y \times Y^* \to \mathbb{R}, \quad (y,
  \lambda) \mapsto \langle y, \lambda\rangle=\Tr\left(\int_0^{\infty}
    y(x) \lambda(dx) \right),
\]
where $\Tr $ denotes the trace. We also define another bilinear map
via
\begin{align}\label{eq:bilinearSd}
  \langle \langle \cdot, \cdot \rangle \rangle: Y \times Y^* \to \mathbb{S}^d, \quad
  (y, \lambda) \mapsto \langle \langle y, \lambda \rangle \rangle =\int_0^{\infty} y(x)\lambda(dx) + \int_0^{\infty} \lambda(dx) y(x).
\end{align}

In the following we summarize the main ingredients of our setting.
For the norm on $\mathbb{S}^d$ we write $\| \cdot \|$, which is given
by $\| u \| = \sqrt{\Tr(u^2)}$ for $u \in \mathbb{S}^d$.

\begin{assumption}\label{ass:weak_existence}
  Throughout this section we shall work under the following
  conditions:
  \begin{enumerate}
  \item We are given an admissible weight function $ \rho$ on $ Y^* $
    (in the sense of Section~\ref{sec:genFeller}) such that
    \[
      \rho(\lambda) = 1+ {\|\lambda\|}_{Y^*}^2, \quad \lambda \in Y^*,
    \]
    where $\|\cdot\|_{Y^*}$ denotes the norm on $Y^*$, which is the
    total variation norm of $ \lambda$.
  \item We are given a closed convex cone $ \mathcal{E} \subset Y^* $
    (in the sequel the cone of $\mathbb{S}^d_+$ valued measures) such
    that $ (\mathcal{E},\rho) $ is a weighted space in the sense of
    Section \ref{sec:genFeller}.  This will serve as statespace of
    \eqref{eq:SPDE_weak}.
  \item Let $ Z \subset Y $ be a continuously embedded subspace.
  \item We assume that a semigroup $ \mathcal{S}^* $ with generator
    $ \mathcal{A}^* $ acts in a strongly continuous way on $ Y^* $ and
    $ Z^* $, with respect to the respective norm
    topologies. 
    Moreover, we suppose that for any matrix $A \in \mathbb{S}^d$ it
    holds that
    \begin{align}\label{eq:semigroupprop}
      \mathcal{S}^*_t(\lambda(\cdot) A+ A \lambda(\cdot))= (\mathcal{S}^*_t\lambda(\cdot)) A+ A (\mathcal{S}^*_t \lambda(\cdot)).
    \end{align}
  \item We assume that $ \lambda \mapsto \mathcal{S}^*_t\lambda $ is
    weak-$*$-continuous on $ Y^* $ and on $ Z^* $ for every
    $ t \geq 0 $ (considering the weak-$*$-topology on both the domain
    and the image space).
  \item We suppose that the (pre-) adjoint operator of
    $ \mathcal{A}^* $, denoted by $\mathcal{A}$ and domain
    $ \operatorname{dom}(\mathcal{A}) \subset Z \subset Y $, generates
    a strongly continuous semigroup on $Z$ with respect to the
    respective norm topology (but \emph{not} necessarily on $ Y $).
  \end{enumerate}
\end{assumption}

To analyze solvability of \eqref{eq:SPDE_weak} we first consider the
following linear deterministic equation
\begin{align}\label{eq:lambda_abstract}
  d \lambda_t(dx) = \mathcal{A}^* \lambda_t(dx) dt + \nu(dx) \beta (\lambda_t(\cdot) )dt+ \beta (\lambda_t(\cdot)) \nu(dx)dt
\end{align}
for $ \lambda_0 \in Y^* $, $ \nu \in Z^*$ and $\beta$ a bounded linear
operator from $Y^* \to \mathbb{S}^d$ which satisfies for
$ A \in \mathbb{S}^d$ and $\lambda \in Y^*$
\begin{align}\label{eq:op}
  \beta(\lambda(\cdot) A +A\lambda(\cdot))= \beta(\lambda(\cdot)) A + A \beta(\lambda(\cdot)).
\end{align}

We denote by $\beta_*: \mathbb{S}^d \to Y$ the adjoint operator
defined via
\[
  \Tr( u \beta(\lambda))= \Tr(\int_0^{\infty} \beta_*(u)(x)
  \lambda(dx))=\langle \beta_*(u), \lambda \rangle, \quad u \in
  \mathbb{S}^d, \, \lambda \in Y^*.
\]

\begin{remark}
  Notice that drift specifications could be more general here, but for
  the sake or readability we leave this direction for the interested
  reader.
\end{remark}

For notational convenience we shall often leave the $dx$ argument away
when writing an (S)PDE of type \eqref{eq:lambda_abstract}
subsequently. Under the following assumptions on $ \mathcal{S}^* $ and
$ \nu \in Z^* $ we can guarantee that \eqref{eq:lambda_abstract} can
be solved on the space $Y^*$ for all times in the mild sense with
respect to the dual norm $\|\cdot\|_{Y^*}$ by a standard Picard
iteration method.

\begin{assumption}\label{ass:semigroup}
  We assume that
  \begin{enumerate}
  \item $ \mathcal{S}^*_t \nu \in Y^* $ for all $ t > 0 $ even though
    $ \nu $ does not necessarily lie in $ Y^* $ itself, but only in
    $ Z^* $;
  \item $ \int_0^t \| \mathcal{S}^*_s \nu \|^2_{Y^*} ds < \infty $ for
    all $ t > 0 $.
  \end{enumerate}
\end{assumption}

For the linear operator $\beta$ as of \eqref{eq:op}, we define
\begin{align} \label{eq:kernel} K(t) := \beta(S_t^{*} \nu),
\end{align}
which will correspond to a kernel in
$L^2_{\text{loc}}(\mathbb{R}_+, \mathbb{S}^d)$ of a Volterra equation.
Define furthermore
$R_K \in L^2_{\text{loc}}(\mathbb{R}_+, \mathbb{S}^d)$ as a
symmetrized version of the resolvent of the second kind (see
e.g.~\cite[Theorem 3.1]{gri:90}) that solves
\begin{align}\label{eq:resolvent}
  K*R_K+R_K * K=K-R_K,
\end{align}
where $K*R_K$ denotes the convolution, i.e.
$K*R_K= \int_0^{\cdot} K(\cdot-s)R_K(s) ds$.

\begin{example}\label{ex:canonical}
  The main example that we have in mind for $\beta$ and for
  $\mathcal{S}^*$, and thus in turn for the kernel $K$, are the
  following specifications:
  \[
    \beta(\lambda) =\int_0^{\infty} \lambda(dx), \quad \mathcal{S}^*_t
    \nu(dx) = e^{-xt}\nu(dx).
  \]
  In this case $K= \int_0^{\infty} e^{-xt}\nu(dx)$ and the adjoint
  operator $\beta_*$ is given by the constant function
  \[
    (\beta_*(u))(x)= u, \quad \text{for all } x \in \mathbb{R}_+.
  \]
\end{example}

\begin{remark}\label{rem:frac} To the semigroup $\mathcal{S}^*_t= e^{-xt}$ of the
  above example, we associate our (main) specification of the space
  $ Z $: 
  let $Z \subset Y$ such that for all $y \in Y$ the map
  \[
    h_y : \overline{ \mathbb{R}}_+ \to \mathbb{S}^d, \quad x \mapsto x
    y(x)
  \]
  lies in $Z$ equipped with the operatornorm, i.e.
$$ 
\| h_y \|_Z =\sqrt{\sup_{ x \geq 0} \| y(x) \| + \sup_{x \geq 0}\| x
  y(x) \| } \text { for } h_y \in Z \, .
$$ 
The corresponding dual space $Z^* \supset Y^*$ is the space of regular
$\mathbb{S}^d$-valued Borel measures $\nu$ on
$\overline{\mathbb{R}}_+$ that satisfy
\[
  \|\int_0^{\infty} (\frac{1}{x} \wedge 1) \nu(dx)\| <\infty \, .
\]
Note that we can specify the components of $\nu$ to be measures of the
form
\[
  \nu_{ij}(dx) = x^{-\frac{1}{2}-H_{ij}}, \quad H_{ij} \in \left(0,
    \frac{1}{2}\right),
\]
which gives rise to fractional kernels
$K_{ij}(t)=\int_0^{\infty} e^{-xt} \nu_{ij}(dx) \approx
t^{H_{ij}-\frac{1}{2}}$. These are in turn main ingredients of rough covariance modeling. 
\end{remark}

\begin{remark}
  In this article we choose to work with state spaces of matrix valued
  measures using the representation of the kernel $K$ as Laplace
  transform of a matrix valued measure $\nu$ as specified in Example
  \ref{ex:canonical}.  We could however perform the same analysis on a
  Hilbert space of forward covariance curves. This corresponds then to
  a multivariate analogon of \cite[Section 5.2]{cuctei:18}.
\end{remark}

\begin{proposition}\label{prop:existenceY*} 
  Under Assumption \ref{ass:semigroup}, there exists a unique mild
  solution of \eqref{eq:lambda_abstract} with values in
  $Y^*$. Additionally, the solution operator is a weak-$*$-continuous
  map $ \lambda_0 \mapsto \lambda_t $, for each $ t > 0 $, and the
  solution satisfies
  \[
    \rho(\lambda_t) \leq C \rho(\lambda_0), \quad \text{for all }
    \lambda_0 \in Y^* \text{ and } t \in [0, \varepsilon]
  \]
  for some positive constants $ C $ and $ \varepsilon $.
\end{proposition}

\begin{remark}
  The unique mild solution of Equation \eqref{eq:lambda_abstract}
  satisfies by means of \eqref{eq:semigroupprop} the variation of
  constants equation
  \[
    \lambda_t = \mathcal{S}^*_t \lambda_0 + \int_0^t (
    \mathcal{S}^*_{t-s} \nu \beta(\lambda_s) + \beta(\lambda_s)
    \mathcal{S}^*_{t-s} \nu )ds,
  \]
  for all $ t \geq 0 $. Applying the linear operator $\beta$ and using
  property \eqref{eq:op}, we obtain a deterministic linear Volterra
  equation of the form
  \begin{equation}\label{eq:detVolt}
    \begin{split}
      \beta( \lambda_t ) &= \beta(\mathcal{S}_t^{*} \lambda_0 )+
      \int_0^t \beta\left( \mathcal{S}_{t-s}^{*} \nu \beta(\lambda_s)
        +\beta(\lambda_s)\mathcal{S}_{t-s}^{*}
        \nu  \right) ds\\
      &=\beta(\mathcal{S}_t^{*} \lambda_0 )+ \int_0^t \left(K(t-s)
        \beta(\lambda_s) +\beta(\lambda_s)K(t-s) \right) ds
    \end{split}
  \end{equation}
  where we have used \eqref{eq:kernel}.
\end{remark}

\begin{proof}
  We follow the arguments of \cite{cuctei:18} and translate the proof
  to the matrix-valued stetting.  We show first the completely
  standard convergence of the Picard iteration scheme with respect to
  the dual norm on $ Y^* $. Define
  \begin{align*}
    \lambda^0_t&= \lambda_0,\\
    \lambda^{n+1}_t&= \mathcal{S}^*_t \lambda_0 + \int_0^t (\mathcal{S}^*_{t-s} \nu )\beta (\lambda^n_s) ds+ \int_0^t\beta (\lambda^n_s)   (\mathcal{S}^*_{t-s} \nu )ds, \quad n \geq 0.
  \end{align*}
  Then, by Assumption \ref{ass:semigroup} (i) each $\lambda^n_t$ lies
  $Y^*$. Consider now
  \begin{align*}
    \|\lambda^{n+1}_t - \lambda^{n}_t\|_{Y^*} &=\| \int_0^t (\mathcal{S}^*_{t-s} \nu )(\beta (\lambda^{n}_s) - \beta (\lambda^{n-1}_s) )ds\\
                                              &\quad + \int_0^t(\beta (\lambda^{n}_s) -\beta (\lambda^{n-1}_s))  (\mathcal{S}^*_{t-s} \nu )ds\|_{Y^*}\\
                                              &\leq 2 \| \beta\|_{\text{op}}\int_0^t  \|\mathcal{S}^*_{t-s}  \nu\|_{Y^*}  \|\lambda^n_s -\lambda^{n-1}_s  \|_{Y^*}   ds,
  \end{align*}
  where $\| \beta\|_{\text{op}}$ denotes the operator norm of $\beta$.
  Assumption \ref{ass:semigroup} (ii) and an extended version of
  Gronwall's inequality see \cite[Lemma 15]{D:99} then yield
  convergence of $(\lambda^n_t)_{n \in \mathbb{N}}$ to some
  $\lambda_t$ with respect to the dual norm $\|\cdot\|_{Y^*}$
  uniformly in $t$ on compact intervals. For details on strongly
  continuous semigroups and mild solutions see \cite{paz:83}.

  Having established the existence of a mild solution of
  \eqref{eq:lambda_abstract} in $Y^*$, consider now the
  $\mathbb{S}^d$-valued process $\beta(\lambda_t)$:
  \begin{equation}
    \begin{split}\label{eq:Volterra_det}
      \beta(\lambda_t) &= \beta(\mathcal{S}_t^{*} \lambda_0 )+
      \int_0^t \beta\left( \mathcal{S}_{t-s}^* \nu \beta(\lambda_s) + \beta(\lambda_s) \mathcal{S}_{t-s}^* \nu  \right) ds ,\\
      &=\beta(\mathcal{S}_t^{*} \lambda_0)+ \int_0^t \left( \beta(\mathcal{ S}_{t-s}^* \nu) \beta(\lambda_s)  + \beta(\lambda_s) \beta(\mathcal{S}_{t-s}^* \nu) \right) ds\\
      &=\beta(\mathcal{S}_t^{*} \lambda_0)+ \int_0^t \left( R_K(t-s)
        \beta(\mathcal{S}_s^{*} \lambda_0) + \beta(\mathcal{S}_s^{*}
        \lambda_0) R_K(t-s)\right) ds
    \end{split}
  \end{equation}  
  where we applied property \eqref{eq:op}.  Remember that $R_K$
  denotes the resolvent of the second kind of
  $ K(t)=\beta(\mathcal{S}_{t}^{*} \nu )$ as introduced in
  \eqref{eq:resolvent} by means of which we can solve the above
  equation in terms of integrals of
  $ t \mapsto \beta(\mathcal{S}_t^* \lambda_0 ) $. Since by assumption
  $ \mathcal{S}^* $ is a weak-$*$-continuous solution operator, the
  map
  $ \lambda_0 \mapsto (t \mapsto \beta(\mathcal{S}^*_t \lambda_0 )) $
  is weak-$*$-continuous as a map from $ Y^* $ to
  $ C(\mathbb{R}_{+},\mathbb{S}^d) $ (with the topology of uniform
  convergence on compacts on $C(\mathbb{R}_{+},\mathbb{S}^d)$). From
  \eqref{eq:Volterra_det} we thus infer that $ \beta(\lambda_t ) $ is
  weak-$*$-continuous for every $ t \geq 0 $, which clearly translates
  to the solution map of Equation \eqref{eq:lambda_abstract}.

  Finally we have to show that the stated inequality for
  $ \rho(\lambda_t) $ holds true on small time intervals
  $ [0,\varepsilon]$. Observe first that for $t \in [0,\varepsilon]$
  \[
    \|\mathcal{S}^*_t \lambda\|_{Y^*}^2 \leq C \|\lambda\|^2_{Y^*}
  \]
  for all $\lambda \in Y^*$ just by the assumption that
  $\mathcal{S}^*_t$ is strongly continuous, for some constant
  $ C \geq 1 $. Furthermore for $t \in [0, \varepsilon]$
  \begin{align*}
    \|\lambda_t\|_{Y^*}^2 &\leq 3( C \|\lambda_0\|^2_{Y^*} +  t\int_0^t \| \mathcal{S}_{t-s}^* \nu \beta(\lambda_s)\|_{Y^*}^2 + t\int_0^t \| \beta(\lambda_s) \mathcal{S}_{t-s}^* \nu \|_{Y^*}^2)\\
                          &\leq 3( C \|\lambda_0\|^2_{Y^*} + 2\varepsilon \|\beta\|^2_{\text{op}} \int_0^t  \|\mathcal{S}^*_{t-s} \nu\|^2_{Y^*} \|\lambda_s \|_{Y^*}^2  ds ).
  \end{align*}
  Consider now the kernel
  $K'(t,s)=6\varepsilon \|\beta\|^2_{\text{op}} \|\mathcal{S}^*_{t-s}
  \nu\|^2_{Y^*} 1_{\{s \leq t\}}$ and denote by $R'$ the resolvent of
  $-K'$, which is nonpositive.  By exactly the same arguments as in
  \cite{cuctei:18}, we then have for $t \in [0,\varepsilon]$
  \[
    \|\lambda_t\|_{Y^*}^2 \leq\widetilde{C} \|\lambda_0\|^2_{Y^*} (1 -
    \int_0^{\varepsilon} R'(s) ds),
  \]
  for some constant $\widetilde{C}$.  This leads to the desired
  assertion due to the definition of $\rho$. From this inequality also
  uniqueness follows in a standard way.
\end{proof}

As our goal is to consider $\mathbb{S}^d_+$-measure valued processes,
we denote by $\mathcal{E}$ the following weak-$*$-closed convex cone
\[
  \mathcal{E}=\{\lambda_0 \in Y^* \, |\, \lambda_0 \text{ is an }
  \mathbb{S}^d_+ \text{ -valued measure on }
  \overline{\mathbb{R}}_+\}.
\]

The next proposition establishes that the solution of
\eqref{eq:lambda_abstract} leaves $\mathcal{E}$ invariant, if the
following assumption holds true:

\begin{assumption}\label{ass:para}
  We assume that
  \begin{enumerate}
  \item $\mathcal{S}^*_t (\mathcal{E}) \subseteq \mathcal{E}$;
  \item $\nu$ is an $\mathbb{S}^d_+$-valued measure;
  \item $\beta(\mathcal{E}) \subseteq \mathbb{S}^d_+$.
  \end{enumerate}
\end{assumption}

\begin{proposition}\label{prop:invariance}
  Let Assumptions \ref{ass:semigroup} and \ref{ass:para} be in
  force. Then the solution of \eqref{eq:lambda_abstract} leaves
  $ \mathcal{E} $ invariant and it defines a generalized Feller
  semigroup on $ (\mathcal{E},\rho) $ by
  $ P_t f(\lambda_0) := f(\lambda_t) $ for all
  $ f \in \mathcal{B}^\rho(\mathcal{E}) $ and $ t \geq 0 $.
\end{proposition}

\begin{proof}
  Consider first the slightly modified equation
  \begin{align}\label{eq:modlambda}
    d \lambda_t(dx) = \mathcal{A}^* \lambda_t(dx) dt +\mathcal{S}_{\varepsilon}^* \nu(dx) \beta (\lambda_t(\cdot) )dt+ \beta (\lambda_t(\cdot)) \mathcal{S}_{\varepsilon}^*\nu(dx)dt
  \end{align}
  for some $\varepsilon >0$.  Then the operator
  $B=\mathcal{S}_{\varepsilon}^* \nu(dx) \beta (\cdot) + \beta (\cdot)
  \mathcal{S}_{\varepsilon}^*\nu(dx)$ is bounded and the associated
  semigroup is given by $P_t^{\varepsilon}=e^{Bt}$. Due to the
  assumptions on $\mathcal{S}^*$, $\nu$ and $\beta$, we have
  $B(\mathcal{E}) \subseteq \mathcal{E}$ implying that
  $P^{\varepsilon}_t(\mathcal{E}) \subseteq \mathcal{E}$ for all
  $t \geq 0$. The Trotter-Kato Theorem (see, e.g., \cite[Theorem
  III.5.8]{engnag:00}) then yields that the semigroup associated to
  \eqref{eq:modlambda} maps $\mathcal{E}$ to itself. This then also
  holds true for the limit when $\varepsilon =0$ by Theorem
  \ref{thm:approximation}.

  Since by Proposition \ref{prop:existenceY*} the solution operator is
  weak-$*$-continuous, we can conclude that
  $\lambda_0 \mapsto f(\lambda_t)$ lies in
  $ \mathcal{B}^\rho(\mathcal{E}) $ for a dense set of
  $ \mathcal{B}^\rho(\mathcal{E}) $ by Theorem
  \ref{theorem:boundedweakcontapprox}. Moreover, it satisfies the
  necessary bound \eqref{eq:markovpsibound} for $ \rho $ and
  \eqref{eq:contint} is satisfied by (norm)-continuity of
  $t \mapsto \lambda_t$. Hence all the conditions of Assumption
  \ref{ass:generic} are satisfied and the solution operator therefore
  defines a generalized Feller semigroup $ (P_t) $ on
  $ \mathcal{B}^\rho(\mathcal{E}) $ by Theorem
  \ref{theorem:strongcontprocess}. This generalized Feller semigroup
  of course coincides with the previously constructed limit.
\end{proof}

By the previous results we can now construct a generalized Feller
process on $ \mathcal{E} $ which jumps up by multiples of
$ \mathcal{S}^*_{\varepsilon}\nu $ for some $\varepsilon \geq 0$ and
with an instantaneous intensity of size $ \beta(\lambda_t) $.  Recall
that $\mathcal{E}_* \subset Y$ denotes the (pre-)polar cone of
$ \mathcal{E}$, that is
\[
  \mathcal{E}_*=\{y \in Y \, |\, y \in C_b(\overline{\mathbb{R}}_+,
  \mathbb{S}_-^d)\}.
\]
Recall the notation from  \eqref{eq:bilinearSd} and define the following set
\begin{align}\label{eq:mathcalD}
  \mathcal{D}=\{ y \in  Y \, |\, y \in \operatorname{dom}(\mathcal{A})~\text{ s.t. } \langle \langle y, \nu \rangle \rangle \text{ is well-defined}\}.
\end{align}

\begin{proposition}\label{prop:epsjumps}
  Let Assumptions \ref{ass:semigroup} and \ref{ass:para} be in
  force. Moreover, let $\mu$ be a finite $\mathbb{S}^d_+$-valued
  measure on $\mathbb{S}^d_+$ such that
  $\int_{\|\xi\| \geq 1}\|\xi\|^2 \|\mu(d\xi)\|< \infty$.  Consider
  the SPDE
  \begin{align}\label{eq:SPDEbase}
    d \lambda_t & = \mathcal{A}^* \lambda_t  dt +\nu \beta(\lambda_t) dt + \beta(\lambda_t) \nu dt + \mathcal{S}^*_\varepsilon \nu dN_t + dN_t \mathcal{S}^*_\varepsilon \nu ,
  \end{align}
  where $(N_t)_{t \geq 0}$ is a pure jump process with jump sizes in
  $\mathbb{S}^d_+$ and compensator
  \[
    \int_0^{\cdot} \int_{\mathbb{S}_+^d} \xi \Tr\left(\beta(\lambda_s)
      \mu(d\xi) \right) ds.
  \]
  \begin{enumerate}
  \item Then for every $ \lambda_0 \in \mathcal{E} $ and
    $ \varepsilon > 0 $ , the SPDE \eqref{eq:SPDEbase} has a solution
    in $ \mathcal{E} $ given by a generalized Feller process
    associated to the generator of \eqref{eq:SPDEbase}.
  \item This generalized Feller process is \emph{also} a
    probabilistically weak and analytically mild solution
    of \eqref{eq:SPDEbase}, i.e.
    \begin{align*}
      \lambda_t & = \mathcal{S}^*_t \lambda_0 ds +\int_0^t
                  \mathcal{S}^*_{t-s}\nu \beta(\lambda_s) ds +
                  \int_0^t\beta(\lambda_s) \mathcal{S}_{t-s}^*\nu ds + \\
                & \quad +\int_0^t\mathcal{S}^*_{t-s+\varepsilon} \nu dN_s+ \int_0^t dN_s
                  \mathcal{S}^*_{t-s+\varepsilon} \nu \, ,
    \end{align*}
    which justifies Equation \eqref{eq:SPDEbase}. In particular for
    every initial value the process $ N $ can be constructed on an
    appropriate probabilistic basis. The stochastic integral is
    defined in a pathwise way along finite variation paths. Moreover, for every family $(f_n)_n \in \mathcal{B}^{\rho}(\mathcal{E})$, $t \mapsto f_n(\lambda_t)$ can be chosen to be c\`agl\`ad for all $n$.
  \item For every $ \varepsilon > 0 $, the corresponding Riccati
    equation $\partial_t y_t=R(y_t)$ with
    $R: \mathcal{D} \cap \mathcal{E}_* \to Y$ given by
    \begin{equation}\label{eq:Riccatisimpl}
      \begin{split}
        R(y) &= \mathcal{A} y + \beta_*\left(\int_0^{\infty} y(x)
          \nu(dx) + \nu(dx) y(x)\right)
        \\
        &\quad + \beta_*\left(\int_{\mathbb{S}^d_+} \left(\exp(
            \langle y , \mathcal{S}^*_{\varepsilon} \nu \xi +\xi
            \mathcal{S}^*_{\varepsilon} \nu \rangle )-1 \right)
          \mu(d\xi)\right),
      \end{split}
    \end{equation}
     admits
    a unique global solution in the mild sense for all initial values
    $ y_0 \in \mathcal{E}_* $.
  \item The affine transform formula holds true, i.e.
    \[
      \mathbb{E}_{\lambda_0}\left[ \exp( \langle y_0, \lambda_t
        \rangle)\right]=\exp(\langle y_t, \lambda_0 \rangle),
    \]
    where $y_t$ solves $\partial_t y_t=R(y_t)$ for all $y_0 \in \mathcal{E}_*$ in the mild sense with
    $R$ given by \eqref{eq:Riccatisimpl}. Moreover
    $y_t \in \mathcal{E}_*$ for all $t \geq 0$.
  \end{enumerate}
\end{proposition}

\begin{proof}
  We assume that $ \nu \neq 0 $, otherwise there is nothing to
  prove. To prove the first assertion we apply Proposition
  \ref{prop:jump_perturbation}.  By Proposition \ref{prop:existenceY*}
  and Proposition \ref{prop:invariance}, the deterministic equation
  \eqref{eq:lambda_abstract} has a mild solution on $\mathcal{E}$
  which -- by Assumption \ref{ass:semigroup} -- defines a generalized
  Feller semigroup $(P_t)_{t\geq 0}$ on
  $ \mathcal{B}^\rho(\mathcal{E}) $.  The operator $A$ in Proposition
  \ref{prop:jump_perturbation} then corresponds to the generator of
  $(P_t)_{t\geq 0}$, i.e.~the semigroup associated to the purely
  deterministic part of \eqref{eq:SPDEbase}. This is a transport
  semigroup and in view of Remark \ref{rem:quasicontractive} we can
  have an equivalent norm with respect to a new weight function
  $ \tilde \rho $ on $\mathcal{B}^{\rho}(\mathcal{E})$, such that
  $ \| P_t \|_{L(B^{\tilde \rho}(\mathcal{E}))} \leq \exp(\omega t) $. Therefore
  we find ourselves in the conditions of Proposition
  \ref{prop:jump_perturbation}.

  Note that by the same arguments as in Proposition
  \ref{prop:invariance} and by applying Theorem
  \ref{theorem:strongcontprocess}, we can prove that
  $(P_t)_{t \geq 0}$ also defines a generalized Feller semigroup on
  $ \mathcal{B}^{\sqrt{\rho}}(\mathcal{E}) $. For the detailed
  proof which translates literally to the present setting we refer to
  \cite{cuctei:18}.
  	
  Finally, we need to verify \eqref{eq:cond1} - \eqref{eq:cond3},
  which read as follows
  \begin{align*}
    \int \ \rho(\lambda + \mathcal{S}^*_{\varepsilon} \nu \xi+ \xi \mathcal{S}^*_{\varepsilon}  \nu) \Tr(\beta(\lambda ) \mu(d\xi)) &\leq M  \rho(\lambda)^2, \\
    \int \sqrt{ \rho}(\lambda + \mathcal{S}^*_{\varepsilon} \nu \xi+ \xi \mathcal{S}^*_{\varepsilon}  \nu) \Tr(\beta(\lambda ) \mu(d\xi)) &\leq M  \rho(\lambda),\\
    \int \Tr( \beta( \lambda ) \mu(d\xi)) &\leq M \sqrt{ \rho(\lambda)} \, ,
  \end{align*}
 which hold true by the second moment condition on $\mu$.
  Concerning \eqref{eq:cond4}, denote as in Remark~\ref{rem:quasicontractive}
  \[
    \tilde{\rho}(\lambda) =\sup_{t\geq 0} \exp(-\omega t) P_t
    \rho(\lambda)  \, .
  \]
In particular we know that $ \rho \leq \tilde \rho$ and it holds that 
$P_tf(x)=f(\psi_t(x))$ where $\psi$ is the solution of \eqref{eq:lambda_abstract} which is linear. Using this together with $ | \sup_t c(t) - \sup_t d(t)| \leq \sup_t |c(t) - d(t) | $ we obtain for some $\widetilde{\omega}$
  \begin{align*}
    & \int  \big | \frac{\tilde{\rho}(\lambda + \mathcal{S}^*_{\varepsilon} \nu \xi+ \xi \mathcal{S}^*_{\varepsilon} \nu)-\tilde{\rho}(\lambda)}{\tilde{\rho}(\lambda)} \big | \Tr( \beta(\lambda) \mu(d\xi))  \\ &\leq\int \big| \frac{\sup_{t \geq 0} \exp(-\omega t) |P_t\rho(\lambda + \mathcal{S}^*_{\varepsilon} \nu \xi + \xi \mathcal{S}^*_{\varepsilon} \nu)-  P_t \rho( \lambda)|}{\tilde\rho(\lambda)} \big | \Tr( \beta(\lambda) \mu(d\xi))\\      
    &\leq\int \big| \frac{\sup_{t \geq 0} \exp(-\omega t) |\rho(\psi_t(\lambda + \mathcal{S}^*_{\varepsilon} \nu \xi +\xi \mathcal{S}^*_{\varepsilon} \nu)) -  \rho(\psi_t (\lambda))|}{\tilde\rho(\lambda)} \big | \Tr( \beta(\lambda) \mu(d\xi))\\   
&= \int \big | \frac{ \sup_{t \geq 0} \exp(-\omega t) ( 2\| \psi_t( \lambda) \|_{Y^*} \; \| \psi_t( \mathcal{S}^*_{\varepsilon} \nu \xi +\xi \mathcal{S}^*_{\varepsilon} \nu )\|_{Y^*} + {\| \psi_t (\mathcal{S}^*_{\varepsilon} \nu \xi+\xi \mathcal{S}^*_{\varepsilon} \nu) \|}_{Y^*}^2 )}{\rho( \lambda )} \big |\\
&\quad \quad \times \Tr( \beta(\lambda) \mu(d\xi)) \leq \widetilde{\omega} \, .
  \end{align*}
  The last inequality holds by the linearity of $\psi$ and the second 
 moment condition on $\mu$.   Proposition
  \ref{prop:jump_perturbation} now allows to conclude that $A+B$,
  where $B$ is given by
  \[
    Bf(\lambda)= \int (f(\lambda + \mathcal{S}^*_{\varepsilon} \nu
    \xi+\xi \mathcal{S}^*_{\varepsilon} \nu ) -f(\lambda))
    \Tr(\beta(\lambda)\mu(d\xi)),
  \]
  generates a generalized Feller semigroup $\widetilde P$ as asserted.

  For (ii), we now construct the probabilistically weak and
  analytically mild solution directly from the properties of the
  generalized Feller process: take
  $ y \in \mathcal{D}$ where $\mathcal{D}$ is defined in \eqref{eq:mathcalD}  and
  consider the $\mathbb{S}^d$-valued martingale
  \begin{equation}\label{eq:mart}
    \begin{split}
      M^y_t & := \langle \langle y, \lambda_t \rangle \rangle - \langle \langle y , \lambda_0 \rangle \rangle - \int_0^t \langle \langle  \mathcal{A} y, \lambda_s \rangle \rangle + \langle \langle  y, \nu \beta(\lambda_s)+ \beta(\lambda_s) \nu \rangle \rangle ds    \\
      & \quad - \int_0^t \int \langle \langle y,
      \mathcal{S}^*_\varepsilon \nu \xi + \xi \mathcal{S}^*_{\varepsilon}
      \nu \rangle \rangle \Tr(\beta( \lambda_s )\mu(d \xi) )ds
    \end{split}
  \end{equation}
  for $ t \geq 0 $ (after an appropriate and possible regularization
  according to Theorem \ref{th:path_properties}).

  Let now $y$ be as above with the additional property that
  $ \langle \langle y , \mathcal{S}^*_{\varepsilon}\nu \xi +\xi
  \mathcal{S}^*_{\varepsilon} \nu\rangle \rangle = \pi \xi + \xi \pi$
  for all $\xi \in \mathbb{S}^d_+$ and some fixed
  $ \pi \in \mathbb{S}^d_+$. For such $ y $ define
  \begin{align}\label{eq:N}
    N^\pi_t = \pi N_t + N_t \pi := M^y_t + \int_0^t \int \langle \langle y, \mathcal{S}^*_\varepsilon \nu \xi
    +\xi \mathcal{S}^*_{\varepsilon} \nu\rangle \rangle \Tr(\beta(\lambda_s) \mu(d \xi)) ds
  \end{align}
  for $ t \geq 0 $, which is a c\`agl\`ad semimartingale. Notice that
  the left hand side only defines $ N^\pi $ and not the more
  suggestive $ \pi N + N \pi $. Then $N^\pi$ does not depend on $y$ by
  construction. Indeed, for all $y_i$ with
  $ \langle \langle y_i , \mathcal{S}^*_{\varepsilon}\nu \xi + \xi
  \mathcal{S}^*_{\varepsilon} \nu\rangle \rangle = \pi \xi + \xi \pi$
  for all $\xi$, $i=1,2$, we clearly have
  \[
    \int_0^t \int \langle \langle y_1-y_2, \mathcal{S}^*_\varepsilon \nu
    \xi+ \xi \mathcal{S}^*_{\varepsilon} \nu \rangle \rangle \Tr(
    \beta(\lambda_s) \mu(d \xi)) ds =0
  \]
  and $ M^{y_1} - M^{y_2} =M^{y_1-y_2}= 0 $ as well.  The latter
  follows from the fact that the martingale $ M^y $ is constant if
  $ \langle \langle y , \mathcal{S}^*_\varepsilon \nu \xi + \xi
  \mathcal{S}^*_\varepsilon \nu \rangle \rangle = 0 $ for all $\xi$,
  since its quadratic variation vanishes in this case.

  Moreover, by the definition of $N^\pi$ in \eqref{eq:N} its
  compensator is given by
  $\int_0^t \int (\pi \xi + \xi \pi) \Tr(\beta(\lambda_s)
  \mu(d\xi))ds$.  Since it is sufficient to perform the previous
  construction for finitely many $ \pi $ to obtain all necessary
  projections, a process $ N $ can be defined such that
  $ N^\pi = \pi N + N \pi $, as suggested by the notation.

  By \eqref{eq:mart} and the very definition of \eqref{eq:N} we obtain
  that
  \begin{align*}
    \langle \langle y, \lambda_t \rangle \rangle & = \langle \langle y , \lambda_0 \rangle \rangle  + \int_0^t \langle \langle \mathcal{A}y , \lambda_s \rangle \rangle ds + \int_0^t \langle \langle y,  \nu \beta(\lambda_s) + \beta(\lambda_s) \nu \rangle \rangle  ds \\
                                                 &\quad + \langle \langle y , \mathcal{S}^*_\varepsilon \nu N_t\rangle \rangle +  \langle \langle y ,N_t \mathcal{S}^*_\varepsilon \nu \rangle \rangle \\
  \end{align*}
  for $ y \in \mathcal{D} $. This analytically
  weak form can be translated into a mild form by standard
  methods. Indeed, notice that the integral is just along a finite
  variation path and therefore we can readily apply variation of
  constants. The last assertion about the c\`agl\`ad property is a consequence of Theorem \ref{th:path_properties} by noting that $\rho(\lambda)$ does not explode. This proves (ii).
		 
  Concerning (iii), note first that we have a unique mild solution to
  \begin{align}\label{eq:linadjoint}
    \partial_t y_t = \mathcal{A}y_t +  \beta_*\left(\int_0^{\infty} y(x) \nu(dx) + \int_0^{\infty} \nu(dx) y(x)\right), \quad y_0 \in Y,
  \end{align}
  since this is the adjoint equation of
  \eqref{eq:lambda_abstract}. For the equation with jumps we proceed
  as in Proposition \ref{prop:existenceY*} via Picard
  iteration. Denote the semigroup associated to \eqref{eq:linadjoint}
  by $\mathcal{S}^{\beta_*}$ and define
  \begin{align*}
    y_t^0&= y_0,\\
    y_t^{n}&=\mathcal{S}^{\beta_*}_ty_0+\int_0^t \mathcal{S}^{\beta_*}_{t-s}\beta_*\left(\int_{\mathbb{S}^d_+}
             \left(\exp(  \langle y_s^{n-1} , \mathcal{S}^*_{\varepsilon} \nu \xi +\xi  \mathcal{S}^*_{\varepsilon} \nu  \rangle )-1  \right) \mu(d\xi)\right) ds.
  \end{align*}
  Moreover, for $t \in [0,\delta]$ for some $\delta >0$ we have by
  local Lipschitz continuity of $x\mapsto \exp(x)$
  \begin{align*}
    \| y_t^{n+1}-y_t^n\|_{Y} &\leq  \|\int_0^t \mathcal{S}^{\beta_*}_{t-s}\beta_* \left(\int_{\mathbb{S}_+^d }( \exp(  \langle y^{n}_s , \mathcal{S}^*_{\varepsilon} \nu \xi \rangle )- \exp(  \langle y^{n-1}_s , \mathcal{S}^*_{\varepsilon} \nu \xi \rangle )  )\mu(d\xi) \right) ds\|_{Y}\\
                             & \leq \int_0^t C \|\mathcal{S}^{\beta_*}_{t-s}\beta_*\|_{\text{op}} \| y_s^{n}-y_s^{n-1}\|_{Y} \left(\int_{\mathbb{S}_+^d} \|\mathcal{S}^*_{\varepsilon} \nu \xi\|_{Y^*}\mu(d\xi) \right) ds.
  \end{align*}
  By an extension of Gronwall's inequality (see \cite[Lemma 15]{D:99})
  this yields convergence of $(y_t^n)_{n\in \mathbb{N}}$ with respect
  to $\|\cdot \|_{Y}$ and hence the existence of a unique local mild
  solution to \eqref{eq:Riccatisimpl} up to some maximal life time
  $t_+(y_0)$.  That $t_+(y_0)= \infty$ for all $y_0 \in \mathcal{E}_*$
  follows from the subsequent estimate
  \begin{align*}
    \| y_t\|_{Y} &= \|\mathcal{S}^{\beta_*}_t y_0 +\int_0^t \mathcal{S}^{\beta_*}_{t-s} \beta_* \left(\int_{\mathbb{S}_+^d}  \left(\exp(  \langle y_s , \mathcal{S}^*_{\varepsilon} \nu \xi + \xi  \mathcal{S}^*_{\varepsilon} \nu \rangle )-1 \right) \mu(d\xi) \right)ds \|_{Y}\\
                 &\leq \|\mathcal{S}^{\beta_*}_t y_0 \|_Y +  \int_0^t \|S^{\beta_*}_{t-s} \beta_*\|_{\text{op}} \left(\int_{\mathbb{S}_+^d} |\exp(\langle y_s , \mathcal{S}^*_{\varepsilon} \nu \xi +\xi \mathcal{S}^*_{\varepsilon} \nu \rangle )-1| \mu(dx) \right)ds\\
                 &\leq \|\mathcal{S}^{\beta_*}_t y_0 \|_Y  +  t \sup_{s \leq t}\|\mathcal{S}^{\beta_*}_{s} \beta_*\|_{\text{op}} \mu(\mathbb{S}_+^d),
  \end{align*}
  where we used
  $|\exp(\langle y , \mathcal{S}^*_{\varepsilon} \nu \xi + \xi
  \mathcal{S}^*_{\varepsilon} \nu \rangle )-1| \leq 1$ for all
  $y \in \mathcal{E}_*$ in the last estimate.
		
  To prove (iv), just note that by the existence of a generalized
  Feller semigroup the abstract Cauchy problem for the initial value
  $ \exp(\langle y_0,.\rangle) $ can be solved uniquely for
  $ y_0 \in \mathcal{E}_* $. Indeed,
  $\mathbb{E}_{\lambda} [\exp(\langle y_0,\lambda_t\rangle)]$ uniquely
  solves
  \[
    \partial_t u(t,\lambda)= A u(t, \lambda), \quad u(0, \lambda) =
    \exp(\langle y_0, \lambda \rangle),
  \]
  where $A$ denotes the generator associated to \eqref{eq:SPDEbase}.
  Setting $u(t,\lambda)=\exp(\langle y_t, \lambda \rangle )$, we have
  \[
    \partial_t u(t,\lambda)=\exp(\langle y_t, \lambda \rangle) R(y_t),
  \]
  where the right hand side is nothing else than
  $A\exp(\langle y_t, \lambda \rangle)$, hence the affine transform formula holds
  true. This also implies that $y_t \in \mathcal{E}_*$ for all
  $t \geq 0$, simply because
  $\mathbb{E}_{\lambda} [\exp(\langle y_0,\lambda_t\rangle)] \leq 1$
  for all $\lambda \in \mathcal{E}$.
\end{proof}

We are now ready to state the main theorem of this section, namely an
existence and uniqueness result for equations of the type
\begin{align}\label{eq:SPDEmain}
  d \lambda_t = \mathcal{A}^* \lambda_t dt + \nu dX_t +  dX_t \nu ,
\end{align}
where $(X_t)_{t \geq 0}$ is an $\mathbb{S}^d_+$-valued pure jump It\^o
semimartingale of the form
\begin{align}\label{eq:X}
  X_t = \int_0^t \beta(\lambda_s)ds+ \int_0^t \int_{\mathbb{S}^d_+} \xi
  \mu^X(d\xi,ds) ,
\end{align}
with $\beta$ specified in \eqref{eq:op} satisfying Assumption
\ref{ass:para} and random measure of the jumps $\mu^X$. Its
compensator satisfies the following condition:

\begin{assumption}\label{ass:comp}
  The compensator of $\mu^X$ is given by
  \[
    \Tr\left(\beta(\lambda_t)\frac{\mu(d\xi)}{\|\xi\| \wedge 1}\right)
  \]
  where $\mu$ is an $\mathbb{S}^d_+$-valued finite measure on
  $\mathbb{S}^d_+$ satisfying
  $\int_{\|\xi\| \geq 1}\| \xi \|^2 \|\mu(d\xi)\| < \infty$.
\end{assumption}

For the formulation of the subsequent theorem we shall need the
following set of Fourier basis elements
\begin{align}\label{eq:Fourier}
  \mathcal{D}=\{ f_y: \mathcal{E} \to [0,1]; \lambda \mapsto \exp( \langle y , \lambda \rangle) \, | \, y \in \mathcal{E}_* \cap \operatorname{dom}(\mathcal{A}) \text{ s.t.~} \langle \langle y, \nu \rangle \rangle \text{ is well defined}\}.
\end{align}

\begin{theorem}\label{thm:main}
  Let Assumptions \ref{ass:semigroup}, \ref{ass:para} and
  \ref{ass:comp} be in force.
  \begin{enumerate}
  \item Then the stochastic partial differential equation
    \eqref{eq:SPDEmain} admits a unique Markovian solution
    $(\lambda_t)_{t\geq0} $ in $ \mathcal{E} $ given by a generalized
    Feller semigroup on $ \mathcal{B}^\rho(\mathcal{E}) $ whose
    generator takes  on the set of Fourier elements
    \begin{align*}
      f_y: \mathcal{E} \to [0,1]; \lambda \mapsto \exp( \langle y , \lambda\rangle)   
    \end{align*}
    for $y \in \mathcal{D} \cap \mathcal{E}_*$ where $\mathcal{D}$ is
    defined in \eqref{eq:mathcalD}
      the form
    \begin{align}\label{eq:genmain}
      Af_y(\lambda)=f_y(\lambda) (\langle \mathcal{A}y, \lambda \rangle+ \langle\mathcal{R}(\langle \langle y, \nu \rangle \rangle), \lambda \rangle),
    \end{align}
 with
    $\mathcal{R}: \mathbb{S}^d_- \to Y$ given by
    \begin{equation}\label{eq:Rtrad}
      \begin{split}
        \mathcal{R}(u) &= \beta_*(u)+
        \beta_*\left(\int_{\mathbb{S}^d_+} \left(\exp( \Tr(u \xi) -1
          \right) \frac{\mu(d\xi)}{\|\xi\| \wedge 1} \right).
      \end{split}
    \end{equation}    
  \item This generalized Feller process is \emph{also} a
    probabilistically weak and analytically mild solution of
    \eqref{eq:SPDEmain}, i.e.
    \[
      \lambda_t = \mathcal{S}^*_t \lambda_0 ds +\int_0^t
      \mathcal{S}^*_{t-s}\nu dX_s + \int_0^t dX_s
      \mathcal{S}_{t-s}^*\nu,
    \]
  This justifies Equation \eqref{eq:SPDEmain}, in particular for
    every initial value the process $ X $ can be constructed on an
    appropriate probabilistic basis. The stochastic integral is
    defined in a pathwise way along finite variation paths. 
Moreover, for every family $(f_n)_n \in \mathcal{B}^{\rho}(\mathcal{E})$, $t \mapsto f_n(\lambda_t)$ can be chosen to be c\`ag for all $n$.    
  \item The affine transform formula is satisfied, i.e.
    \[
      \mathbb{E}_{\lambda_0}\left[ \exp( \langle y_0, \lambda_t
        \rangle)\right]=\exp(\langle y_t, \lambda_0 \rangle),
    \]
    where $y_t$ solves $\partial_t y_t=R(y_t)$ for all
    $y_0 \in \mathcal{E}_*$ and $t >0$ in the mild sense with
    $R: \mathcal{D} \cap \mathcal{E}_* \to Y$ given by
    \begin{align}\label{eq:Riccatimain}
      R(y) = \mathcal{A} y + \mathcal{R}( \langle \langle y , \nu \rangle\rangle) 
    \end{align}
    with $\mathcal{R}$ defined in \eqref{eq:Rtrad}.  Furthermore,
    $y_t \in \mathcal{E}_*$ for all $t \geq 0$.
  \item For all $\lambda_0 \in \mathcal{E}$, the corresponding
    stochastic Volterra equation, $V_t:= \beta(\lambda_t )$, given by
    \begin{equation}\label{eq:Voltrep}
      \begin{split}
        V_t= \beta( \lambda_t ) &=\beta(\mathcal{S}_t^* \lambda_0) +
        \int_0^t\beta(\mathcal{S}_{t-s}^*\nu) dX_s + \int_0^t dX_s \beta(\mathcal{S}_{t-s}^*\nu)\\
        & =h(t) + \int_0^t K(t-s)dX_s +\int_0^t dX_s K(t-s)
      \end{split}
    \end{equation}
    with $h(t)=\beta( \mathcal{S}_t^* \lambda_0 )$ admits a 
    probabilistically weak solution with c\`ag trajectories.
  \item[(v)] The Laplace transform of the Volterra equation $V_t$ is
    given by
    \begin{align}\label{eq:Voltlaplace}
      \mathbb{E}_{\lambda_0}\left[\exp\left(\Tr(u V_t)\right)\right]=\exp\left(\Tr(u h(t))+ \int_0^t \Tr (\mathfrak{R}(\psi_s) h(t-s) ) ds\right),
    \end{align}
    where $h(t)=\beta(\mathcal{S}_t^*\lambda_0 )$,
    $\mathfrak{R}: \mathbb{S}^d_- \to \mathbb{S}^d_-,\, u \mapsto
    \mathfrak{R}(u)= u + \int_{\mathbb{S}_+^d} (e^{\Tr(u\xi)}-1 )
    \frac{\mu(d\xi)}{\| \xi \| \wedge 1}$ and $\psi_t$ solves the
    matrix Riccati Volterra equation
    \[
      \psi_t=u K(t)+\int \mathfrak{R}(\psi_s) K(t-s) ds, \quad t >0.
    \]
     Hence the solution of the stochastic Volterra equation in
    \eqref{eq:Voltrep} is unique in law.
  \end{enumerate}
\end{theorem}

\begin{remark}\label{rem:comment}
One essential point here is that we
  loose the c\`agl\`ad property as stated  in Proposition \ref{prop:epsjumps} (ii) when we let $ \varepsilon $ of $\mathcal{S}_{\varepsilon}$  tend to
  zero. As long as the kernel $ K $ has a singularity at $ t = 0 $ it
  is impossible to preserve finite growth bounds \emph{with}
  $ M = 1 $, as
  $ \varepsilon \to 0 $, but we get c\`ag versions (compare with the second conclusion in Theorem \ref{th:path_properties} and Remark \ref{rem:pathproperties}).
\end{remark}
		
\begin{remark}
  Note that for $\beta$ as of Example \ref{ex:canonical} the above
  equations simplify considerably. In particular $\beta_*$ in
  \eqref{eq:Rtrad} is simply the identity.
\end{remark}
		
\begin{proof}
  We apply Theorem \ref{thm:approximation_gen} and consider a sequence
  of generalized Feller semigroups $(P^n)_{n \in \mathbb{N}}$ with
  generators $A^n$ corresponding to the solution $\lambda^n$ of
  \eqref{eq:SPDEbase} for $\varepsilon =\frac{1}{n}$, and compensator
  \[
    \Tr\left(\beta(\lambda^n_t)\frac{1_{\{\|\xi\| > \frac{1}{n}\}}
        \mu(d\xi)}{\|\xi\| \wedge 1}\right), \quad n \in \mathbb{N}.
  \]
  Let us first establish a uniform growth bound for this sequence. To
  this end denote
  \[
    F^n(d\xi):=\frac{1_{\{\|\xi\| > \frac{1}{n}\}} \mu(d\xi)}{\|\xi\|
      \wedge 1}.
  \]
  Note that for the solution of \eqref{eq:SPDEbase}, we have due to
  Proposition \ref{prop:epsjumps} (ii) the following estimate for
  $t \in [0,T]$ for some fixed $T>0$
  \begin{align*}
    \mathbb{E}[\| \lambda^{n}_t\|^2_{Y^*}] &\leq 5\|\mathcal{S}^*_t \lambda_0\|^2_{Y^*}+ 10t\int_0^t \|\mathcal{S}^*_{t-s}\nu \|^2_{Y^*} \|\beta\|^2_{\text{op}} \mathbb{E}[\|\lambda^{n}_s\|^2_{Y^*}]ds \\
                                           &\quad + 10\mathbb{E} [\|\int_0^t \mathcal{S}^*_{t-s+\frac{1}{n}} \nu  dN_s -\int_0^t \int \mathcal{S}^*_{t-s+\frac{1}{n}} \nu \xi \Tr(\beta( \lambda^{n}_s ) F^n(d\xi) )ds\|_{Y^*}^2]\\
                                           &\quad +  10\mathbb{E} [\|\int_0^t dN_s\mathcal{S}^*_{t-s+\frac{1}{n}} \nu   -\int_0^t \int  \xi \mathcal{S}^*_{t-s+\frac{1}{n}} \nu \Tr(\beta( \lambda^{n}_s ) F^n(d\xi) )ds)\|_{Y^*}^2]\\
                                           &\quad + 10\mathbb{E} [\|\int_0^t \int \mathcal{S}^*_{t-s+\frac{1}{n}} \nu \xi \Tr(\beta( \lambda^{n}_s ) F^n(d\xi) )  ds\|_{Y^*}^2]\\
                                           &\quad +  10\mathbb{E} [\|\int_0^t \int\xi  \mathcal{S}^*_{t-s+\frac{1}{n}} \nu  \Tr(\beta( \lambda^{n}_s ) F^n(d\xi) )  ds\|_{Y^*}^2].
  \end{align*}
  As a consequence of It\^o's isometry the martingale part can be
  estimated by
  \begin{align*}
    &\mathbb{E} [\|\int_0^t \mathcal{S}^*_{t-s+\frac{1}{n}} \nu  dN_s -\int_0^t \int \mathcal{S}^*_{t-s+\frac{1}{n}} \nu \xi \Tr(\beta( \lambda^{n}_s ) F^n(d\xi) )ds\|_{Y^*}^2]\\
    &\quad \leq \mathbb{E} [\|\int_0^t \int  \|\mathcal{S}^*_{t-s+\frac{1}{n}} \nu\|_{Y^*}^2  \| \xi\|^2 \Tr(\beta( \lambda^{n}_s ) F^n(d\xi) )ds]\\
    &\quad \leq\int  \| \xi\|^2 \|F^n(d\xi)\| \int_0^t \|\mathcal{S}^*_{t-s+\frac{1}{n}} \nu\|_{Y^*}^2  \|\beta\|_{\text{op}} \mathbb{E} [ \|\lambda^{n}_s \|_{Y^*}]  ds\\
    &\quad \leq
      \left( \int_{\| \xi\| \leq 1}\| \mu(d\xi)\| + \int_{\| \xi \| > 1} \|\xi\|^2 \| \mu (d\xi)\|\right) 
      \int_0^t \|\mathcal{S}^*_{t-s+\frac{1}{n}} \nu\|_{Y^*}^2  \|\beta\|_{\text{op}} \mathbb{E} [ \|\lambda^{n}_s \|_{Y^*}]  ds\\
    &\quad \leq
      \widetilde{C}
      \int_0^t \|\mathcal{S}^*_{t-s+\frac{1}{n}} \nu\|_{Y^*}^2  \|\beta\|_{\text{op}} \mathbb{E} [ \|\lambda^{n}_s \|_{Y^*}]  ds\\
    &\quad \leq \widetilde{C} K \int_0^t \|\mathcal{S}^*_{t-s+\frac{1}{n}} \nu\|_{Y^*}^2  (1 +\|\beta\|^2_{\text{op}} \mathbb{E} [ \|\lambda^{n}_s \|^2_{Y^*}] ) ds
  \end{align*}
  where
  $\widetilde{C}= \left( \int_{\| \xi\| \leq 1}\| \mu(d\xi)\| +
    \int_{\| \xi \| > 1} \|\xi\|^2 \| \mu (d\xi)\|\right)$ and $K$
  some other constant.  Moreover, for the last terms we have
  \begin{align*}
    & \mathbb{E} [\|\int_0^t \int \mathcal{S}^*_{t-s+\frac{1}{n}} \nu \xi \Tr(\beta( \lambda^{n}_s ) F^n(d\xi) )  ds\|_{Y^*}^2]                                                 \\
    &\quad \leq      t \int_0^t \| \mathcal{S}^*_{t-s+\frac{1}{n}} \nu\|_{Y^*}^2  \mathbb{E}[\| \int \xi \Tr(\beta( \lambda^{n}_s ) F^n(d\xi) ) \|^2 ]ds    \\
    &\quad \leq     2t \int_0^t \| \mathcal{S}^*_{t-s+\frac{1}{n}} \nu\|_{Y^*}^2  \mathbb{E}[\| \int_{\| \xi\| \leq 1} \xi \Tr(\beta( \lambda^{n}_s ) F^n(d\xi) ) \|^2 + \|\int_{\| \xi\| \geq 1} \xi \Tr(\beta( \lambda^{n}_s ) F^n(d\xi) )\|^2 ]ds     \\
    &\quad \leq     2t \int_0^t \| \mathcal{S}^*_{t-s+\frac{1}{n}} \nu\|_{Y^*}^2     \|\beta\|^2_{\text{op}}\mathbb{E}[ \|\lambda^{n}_s\|^2_{Y^*}] \int \|\mu (d\xi)\| \left( \int_{\| \xi\| \leq 1}\| \mu(d\xi)\| + \int_{\| \xi \| > 1} \|\xi\|^2 \| \mu (d\xi)\|\right)\\
    &\quad \leq    2t  {\widehat {C}} \int_0^t \| \mathcal{S}^*_{t-s+\frac{1}{n}} \nu\|_{Y^*}^2     \|\beta\|^2_{\text{op}}\mathbb{E}[ \|\lambda^{n}_s\|^2_{Y^*}] 
  \end{align*}         
  where ${\widehat{C}}= \int \|\mu (d\xi)\| \widetilde{C}$.  Putting
  this together, we obtain
  \begin{align*}
    \mathbb{E}[\| \lambda^{n}_t\|^2_{Y^*}]      
    &\leq C_0 \|\lambda_0\|^2_{Y^*}  + 10t\int_0^t \|\mathcal{S}^*_{t-s}\nu \|^2_{Y^*} \|\beta\|^2_{\text{op}} \mathbb{E}[\|\lambda^{n}_s\|^2_{Y^*}]ds\\
    &\quad +20 \widetilde{C}K
      \int_0^t \|\mathcal{S}^*_{t-s+\frac{1}{n}} \nu\|_{Y^*}^2  ds\\
    &\quad + 20 (\widetilde{C}K + 2t  {\widehat {C}})\int_0^t \| \mathcal{S}^*_{t-s+\frac{1}{n}} \nu\|_{Y^*}^2     \|\beta\|^2_{\text{op}}\mathbb{E}[ \|\lambda^{n}_s\|^2_{Y^*}] \\
    &\leq C_0  \|\lambda_0\|^2_{Y^*} + C_1\int_0^t  \|\mathcal{S}^*_{t-s} \nu \|_{Y^*}^2 ds + C_2 \int_0^t \|\mathcal{S}^*_{t-s} \nu\|^2_{Y^*}\mathbb{E}[ \|\lambda^{n}_s\|^2_{Y^*}] ds
  \end{align*}
  where $C_0$ and $C_2$ depend on $T$. We use
  $\|S_t^*\lambda_0\|^2 \leq C_0 \|\lambda_0\|^2$ for $t \in [0,T]$,
  as well as
  $ \|\mathcal{S}^*_{t-s+\frac{1}{n}} \nu\|_{Y^*} \leq C
  \|\mathcal{S}^*_{t-s} \nu\|_{Y^*}$ for some constant $C$ and all
  $n \in \mathbb{N}$ due to strong continuity. Exactly by the same
  arguments as in the proof of Proposition \ref{prop:existenceY*} , we
  thus obtain for $t \in [0,T]$ for some fixed $T$
  \[
    \mathbb{E}[\|\lambda_t\|^2_{Y^*}] \leq
    \widetilde{C}(\|\lambda_0\|^2_{Y^*}+1) (1-\int_0^t R'(s), ds),
  \]
  where $R'$ denotes the resolvent of
  $-C_2\|\mathcal{S}^*_{t-s}\nu \|_{Y^*} $. Hence,
  $\mathbb{E}[\rho(\lambda_t)]\leq C \rho(\lambda_0)$ for
  $t \in [0,T]$. From this the desired uniform growth bound
  $\|P_t\|_{L(\mathcal{B}^{\rho}(\mathcal{E}))} \leq M \exp(\omega t)$
  for some $M \geq 1 $ and $\omega \in \mathbb{R}$ follows.

  For the set $D$ as of Theorem \ref{thm:approximation_gen} we here
  choose Fourier basis elements of the form
  \begin{align}\label{eq:Fourier1}
    f_y: \mathcal{E} \to [0,1]; \lambda \mapsto \exp( \langle y , \lambda \rangle) 
  \end{align}
  such that $y \in \mathcal{E}_*$ and
  $\lambda \mapsto \exp( \langle y , \lambda \rangle)$ lies in
  $\cap_{n \geq 1} \operatorname{dom}(A^n)$, whose span is dense,
  whence (i) of Theorem \ref{thm:approximation_gen}. Here, $A^n$
  denotes the generator corresponding to \eqref{eq:SPDEbase} with
  $\varepsilon =\frac{1}{n}$ and $\mu$ replaced by $F^n$.  We now
  equip $\operatorname{span}(D)$ with the uniform norm
  $\| \cdot\|_{\infty}$ and verify Condition (ii), i.e.~we check
  \begin{align}\label{eq:CondII}
    \|A^n P^m_u f_y -A^mP^m_u f_y\|_{\rho} \leq \|f_y\|_{\infty}a_{nm}
  \end{align}
  for all $0 \leq u \leq t$ with $a_{nm} \to 0$ as $n,m \to \infty$,
  and possibly depending on $y$.  Note that
  \[
    A^nf_y(\lambda)=\langle R^n(y), \lambda\rangle f_y(\lambda),
  \]
  where $R^n$ corresponds to \eqref{eq:Riccatisimpl} for
  $\varepsilon =\frac{1}{n}$ and $\mu$ replaced by $F^n$. As $P^n$
  leaves $D$ invariant for all $n \in \mathbb{N}$ by Proposition
  \ref{prop:epsjumps} (iv), we have
  \begin{align*}
    &\frac{|A^n P^m_u f_y(\lambda) -A^m P^m_u f_y(\lambda)|}{\rho(\lambda)}\\
    &\quad\leq\frac{f_{y^m_u}(\lambda)}{\rho(\lambda)}\Bigg(\beta_{*}\Big(\int_{\mathbb{S}^d_+} \underbrace{\exp(\langle y^m_u, \mathcal{S}^*_{\frac{1}{m}} \nu\xi + \xi \mathcal{S}^*_{\frac{1}{m}} \nu \rangle)1_{\{\| \xi\| \geq \frac{1}{n}\}}}_{:=b_{nm}(\xi)}\\
    &\quad \quad \quad \quad \quad \quad \quad \times\underbrace{|\exp(\langle y^m_u, (\mathcal{S}^*_{\frac{1}{n}}\nu -\mathcal{S}^*_{\frac{1}{m}} \nu ) \xi+ \xi (\mathcal{S}^*_{\frac{1}{n}}\nu -\mathcal{S}^*_{\frac{1}{m}} \nu ) \rangle) -1| }_{\widetilde{a}_{nm}^1(\xi)}\frac{\mu(d\xi )}{\|\xi \| \wedge 1}\Big)\\
    &\quad \quad +\beta_*\Big( \underbrace{\int_{\mathbb{S}^d_+} \exp(\langle y^m_u, \mathcal{S}^*_{\frac{1}{m}} \nu\xi + \xi \mathcal{S}^*_{\frac{1}{m}} \nu \rangle)-1) | 1_{\{\| \xi\| \geq \frac{1}{n}\}}-  1_{\{\| \xi\| \geq \frac{1}{m}\}}|\frac{\mu(d\xi )}{\|\xi \| \wedge 1}\Big)}_{\widetilde{a}^2_{nm}}\Bigg).
  \end{align*}
  Here, $y^m_u$ denotes the solution of $\partial_t y^m_u=R^m(y^m_t)$
  at time $u$ with $y_0=y$. Moreover $\widetilde{a}^1_{nm}(\xi)$ and
  $\widetilde{a}^2_{nm}$ can be chosen uniformly for all $u \leq t$
  and tend to $0$ as $n,m \to \infty$.  This is possible since for the
  chosen initial values $ y $ we obtain that $y^m_u$ is bounded on
  compact intervals in time uniformly in $ m $ (see \cite{cuctei:18}
  for details).  This together with dominated convergence for the
  first term (note that $b_{nm}(\xi) \widetilde{a}^1_{nm}(\xi) $ can
  be bounded by $\| \xi\| \wedge 1$) we thus infer
  \eqref{eq:CondII}. The conditions of Theorem
  \ref{thm:approximation_gen} are therefore satisfied and we obtain a
  generalized Feller semigroup whose generator is given by
  \eqref{eq:genmain}.
  
  For the second assertion we proceed as in the proof of Proposition
  ~\ref{prop:epsjumps}, the proof of the existence of $ X $ can be
  transferred verbatim. However, 
  one looses the existence of c\`agl\`ad paths of $f_n(\lambda)$ due to the  possible lack of
  finite mass of $ \nu $. Here, we only obtain c\`ag  trajectories (compare with Remark \ref{rem:pathproperties} and Remark \ref{rem:comment}).

  Concerning the third assertion, the affine transform formula follows
  simply from the convergence of the semigroups $P^n$ as asserted in
  Theorem \ref{thm:approximation_gen} by setting
  $y_t= \lim_{n \to \infty} y_t^{n}$, where $y_t^{n}$ solves
  $\partial_t y_t^{n}=R^{n}(y_t^{n} )$ in the mild sense with $R^{n}$
  given again by \eqref{eq:Riccatisimpl} with
  $\varepsilon =\frac{1}{n}$ and $\mu$ replaced by $F^n$. Since
  $\exp(\langle y_t, \lambda \rangle)$ is then also the unique
  solution of the abstract Cauchy problem for initial value
  $ \exp(\langle y_0,\lambda\rangle) $, i.e.  it solves
  \[
    \partial_t u(t,\lambda)= A u(t, \lambda), \quad u(0, \lambda) =
    \exp(\langle y_0, \lambda \rangle),
  \]
  where $A$ denotes the generator \eqref{eq:genmain}, we infer that
  $y_t$ satisfies $\partial_t y_t=R(y_t)$ with $R$ given by
  \eqref{eq:Riccatimain}. This is because
  $A\exp(\langle y_t, \lambda \rangle)=\exp(\langle y_t, \lambda
  \rangle)R(y_t)$.

  The fourth claim follows from statement (ii), property \eqref{eq:op}
  and the definition of $K$ in \eqref{eq:kernel}.
  
  Finally to prove (v), note that due to (iv) and the definition of
  the adjoint operator $\beta_*$ we have
  \[
    \Tr(u V_t)= \Tr(u \beta(\lambda_t))= \langle \beta_*(u),
    \lambda_t\rangle.
  \]
  Statement (iii) therefore implies that
  \[
    \mathbb{E}[e^{\Tr(u V_t)}]= e^{\langle y_t, \lambda_0\rangle},
  \]
  where the mild solution of $y_t$ can be expressed by
  \begin{align}\label{eq:yt}
    y_t= \mathcal{S}_t \beta_*(u) + \int_0^t \mathcal{S}_{t-s} \mathcal{R}(\langle \langle y_s, \nu \rangle \rangle)ds.
  \end{align}
  Hence, by definition of $\mathcal{R}$, $\mathfrak{R}$ and $h$, we
  find
  \begin{equation}\label{eq:Voltric}
    \begin{split}
      \langle y_t, \lambda_0\rangle&= \langle \mathcal{S}_t \beta_*(u) + \int_0^t \mathcal{S}_{t-s} \mathcal{R}(\langle \langle y_s, \nu \rangle \rangle)ds, \lambda_0\rangle\\
      &=\Tr(u\beta(\mathcal{S}^*_t\lambda_0))+ \int_0^t \Tr(\mathfrak{R}(\langle \langle y_s, \nu \rangle \rangle)\beta(\mathcal{S}^*_{t-s}\lambda_0)) ds\\
      &=\Tr(u h(t))+ \int_0^t \Tr(\mathfrak{R}(\langle \langle y_s,
      \nu \rangle \rangle)h(t-s))ds
    \end{split}
  \end{equation}
  From this and \eqref{eq:yt} it is easily seen that we can replace
  $\langle \langle y_s, \nu \rangle \rangle$ in \eqref{eq:Voltric} by
  a solution of the following Volterra Riccati equation
  \[
    \psi_t= u K(t)+ \int_0^t \mathfrak{R}(\psi_s)K(t-s).
  \]
  Note that we do not need to symmetrize here since we apply the trace
  and $h$ is symmetric. This proves the assertion.
\end{proof}

The following example illustrates how a multivariate Hawkes process
can easily be defined by means of \eqref{eq:X}.

\begin{example}\label{ex:Hawkes}
  Let $\beta$ and $\mathcal{S}^*$ be as of Example \ref{ex:canonical}.
  Define $\mu_{ii}(d\xi)=\delta_{e_{ii}}(d\xi)$ and $\mu_{ij}=0$ for
  $i\neq j$. Then the Volterra equation as of \eqref{eq:Voltrep} is
  given by
  \begin{equation*}
    \begin{split}
      V_t&= \int_0^{\infty}e^{-xt}\lambda_0(dx)
      + \int_0^t (K(t-s)V_s + V_s K(t-s)) ds \\
      &\quad +\int_0^t K(t-s) dN_s + \int_0^t dN_s K(t-s).
    \end{split}
  \end{equation*}
  Only the diagonal components of the matrix valued process $N$ jump and we can define 
  $\widehat{N}:= \diag(N)$ which is a process with
  values in $\mathbb{N}_0^d$. Its components jump by one and the
  compensator of $N_{ii}=\widehat{N}_i $ is given by
  $\int_0^{\cdot}V_{s,ii} ds$, which justifies the name multivariate
  Hawkes process.  Note that the components of $V$ are not independent
  if $\nu$ and in turn $K$ is not diagonal.
\end{example}

\section{Squares of matrix valued Volterra OU processes}\label{sec:square-OU}

As in the finite dimensional setting squares of Gaussian processes
provide us with important process classes for financial and
statistical modeling. In this section we outline this program in
utmost generality from a stochastic and analytic point of view. In
particular we consider \emph{continuous affine Volterra type
  processes} on $\mathbb{S}_+^d$, which we construct as squares of
matrix-valued Volterra Ornstein-Uhlenbeck (OU) processes (see Remark
\ref{rem:squareOU}). Following the finite dimensional analogon \cite{B:91}, we start by considering
matrix measure-valued OU-processes of the form
\begin{align}\label{eq:OU}
  d\gamma_t(dx)= \mathcal{A}^*\gamma_t(dx)dt+ dW_t \nu(dx), \quad \gamma_0 \in Y^*(\mathbb{R}^{n \times d}).
\end{align}
The underlying Banach space, denoted by
$Y^*(\mathbb{R}^{n \times d})$, is the space of finite
$\mathbb{R}^{n \times d}$-valued regular Borel measures on the
extended half real line
$\overline{\mathbb{R}}_+:=\mathbb{R}_+ \cup \{\infty\}$.  Together
with
\[
  \rho(\gamma) = 1+ {\|\gamma\|}_{Y^*(\mathbb{R}^{n \times d})}^2,
  \quad \gamma \in Y^*(\mathbb{R}^{n \times d}),
\]
where $\| \cdot\|_{Y^*(\mathbb{R}^{n \times d})}$ denotes the total
variation norm, this becomes a weighted space.  Moreover,
$\mathcal{A}^*$ is the generator of a strongly continuous semigroup
$\mathcal{S}^*$ on $Y^*(\mathbb{R}^{n \times d})$, which satisfies a
property analogous to~\eqref{eq:semigroupprop}, i.e., for elements
$A \in \mathbb{R}^{n \times d}$ it holds that
\begin{align}\label{eq:semigrouppropnonsym}
  \mathcal{S}^*_t(\gamma(\cdot) A^{\top})= (\mathcal{S}^*_t\gamma(\cdot)) A^{\top}
  \quad \text{ and } \quad \mathcal{S}^*_t(A\gamma^{\top}(\cdot) )=
  A(\mathcal{S}^*_t\gamma(\cdot))^{\top}.
\end{align}
The process $W$ is a $n\times d$ matrix of Brownian motions and
$\nu \in Y^*=:Y^*(\mathbb{S}^d)$ or $Z^*$, as defined in Section
\ref{sec:markovianlift_abstract} such that Assumption
\ref{ass:semigroup} holds true.  The predual space denoted by
$Y(\mathbb{R}^{n \times d})$ is given by
$C_{b}(\overline{\mathbb{R}}_+, \mathbb{R}^{n \times d})$ functions,
where we fix the pairing $\langle \cdot, \cdot \rangle $ as follows
\[
  \langle \cdot, \cdot \rangle: Y(\mathbb{R}^{n \times d}) \times
  Y^*(\mathbb{R}^{n \times d}) \to \mathbb{R}, \quad (y, \gamma)
  \mapsto \langle y, \gamma \rangle=\Tr\left(\int_0^{\infty}
    y^{\top}(x) \gamma(dx) \right) \, .
\]
Again $\Tr $ denotes the trace. We assume that all relevant properties
from Assumption \ref{ass:weak_existence} are translated to the current
setting.

\begin{remark}
  Observe the analogy to the process $\gamma$ defined in the
  introduction. If $\mathcal{A}^*=0$ and $\nu$ is supported on a
  finite space with $k$ points, then \eqref{eq:OU} is exactly the
  process from the introduction.
\end{remark}

\begin{proposition}\label{eq:propOU}
  For every $ \gamma_0 \in Y^*(\mathbb{R}^{n \times d})$ the SPDE
  \eqref{eq:OU} has a solution given by a generalized Feller semigroup
  on $\mathcal{B}^{\rho}(Y^*(\mathbb{R}^{n \times d}))$ associated to
  the generator of \eqref{eq:OU}. The mild formulation directly yields
  a stochastically strong solution
  \[
    \gamma_t (dx) = S_t^* \gamma_0 (dx) + \int_0^t dW_s S_{t-s}^* \nu
    (dx) \,
  \]
  where order matters, i.e. the matrix Brownian increment is applied
  to $ S_{t-s}^* \nu (dx) $ on the left. The integral is understood in
  the weak sense, i.e.~after pairing with
  $ y \in Y(\mathbb{R}^{n \times d}) $.
\end{proposition}

\begin{proof}
  The construction of the generalized Feller process can be done by
  jump approximation of the Brownian motion similarly as in
  \cite[Theorem 4.16]{cuctei:18}. Notice here that we consider the
  process on the whole space $Y^*(\mathbb{R}^{n \times d})$. So no
  issues with state space constraints occur.

  The right hand side of the stochastically strong formulation defines
  -- after pairing with $ y \in Y(\mathbb{R}^{n \times d}) $ -- almost
  surely a continuous linear functional with value
  \[
    \langle y , S_t^* \gamma_0 \rangle + \int_0^t \langle y , dW_s
    S_{t-s}^* \nu \rangle \, ,
  \]
  since the integrand of the stochastic integral is deterministic and
  in $ L^2 $ for each $ t \geq 0 $.
\end{proof}

In order to define the actual process of interest, we need to
introduce some further notations: for elements in
$\gamma \in Y^{\ast}(\mathbb{R}^{n \times d})$ we define
\[
  (\gamma \widehat{\otimes} \gamma)(\cdot,
  \cdot):=\gamma^{\top}(\cdot) \gamma(\cdot).
\]
The corresponding \emph{contracted}, i.e.~one matrix multiplication is
performed, algebraic tensor product is denoted by
$Y^{\ast}(\mathbb{R}^{n\times d})
\widehat{\otimes}Y^{\ast}(\mathbb{R}^{n\times d})$ and we set
\begin{align}\label{eq:statespacesquare}
  \widehat{\mathcal{E}}:=\big \{ \gamma \widehat{\otimes} \gamma \in Y^{\ast}(\mathbb{R}^{n\times d}) \widehat{\otimes}Y^{\ast}(\mathbb{R}^{n\times d}) \big\}.
\end{align}
This corresponds to the space of finite $\mathbb{S}_+^{d}$-valued,
rank $n$, product measures on
$\overline{\mathbb{R}}_+ \times \overline{\mathbb{R}}_+$. We shall
introduce a particular dual topology on $ \widehat{\mathcal{E}} $,
namely $ \sigma( \widehat{\mathcal{E}},Y \otimes Y) $, where the
corresponding pairing is given by
\begin{align*}
  (y_1 \otimes y_2, \gamma_1\widehat{\otimes} \gamma_2) & \mapsto \langle y_1 \widehat{\otimes} y_2, \gamma_1 \widehat{\otimes} \gamma_2\rangle \\
                                                        &=\Tr\left(\int_0^{\infty}   y_1^{\top}(x_1)y_2(x_2) \gamma_1^{\top}(dx_1)\gamma_2(dx_2) \right) \, .
\end{align*}
We denote the pre-dual cone by
\begin{align}\label{eq:polarsquare}
  -\widehat{\mathcal{E}}_*= \big \{ y \widehat{\otimes} y \in Y(\mathbb{R}^{n\times d}) \widehat{\otimes} Y(\mathbb{R}^{n\times d}) \big \} \, ,
\end{align}
where we use again the contracted algebraic tensor product
corresponding to the following matrix multiplication of
$\mathbb{R}^{n\times d}$ valued functions
\[
  (y\widehat{\otimes} y)(\cdot,\cdot)= y^{\top}(\cdot) y(\cdot),\quad
  y \in Y(\mathbb{R}^{n \times d}) \, .
\]
The minus on the left hand side of \eqref{eq:polarsquare} is to obtain
elements in the polar cone.

Let us now define the actual process of interest, namely
\begin{align}\label{eq:lambda}
  \lambda_t(dx_1, dx_2):= \gamma_t^{\top}(dx_1)\gamma_t(dx_2)= \gamma_t(dx_1) \widehat{\otimes} \gamma_t(dx_2).
\end{align}
Note again the analogy to the Wishart process $\lambda$ defined in the
introduction. The process \eqref{eq:lambda} clearly takes values in
$\widehat{\mathcal{E}}$ as defined in \eqref{eq:statespacesquare}.  We
will now show that we can define a Volterra type process by
considering projections on $\mathbb{S}_+^d$.  Applying It\^o's
formula, we see that $\lambda_t(dx_1, dx_2)$ satisfies the following
equation
\begin{equation}\label{eq:lambda1}
  \begin{split}
    d\lambda_t(dx_1, dx_2)&= \left(\mathcal{A}_1^*\lambda_t(dx_1, dx_2) + \mathcal{A}_2^*\lambda_t(dx_1, dx_2) + n \nu(dx_1)\nu(dx_2) \right)dt\\
    &\quad + \nu(dx_1) dW_t^{\top} \gamma_t(dx_2) +
    \gamma_t(dx_1)^{\top} dW_t \nu(dx_2),
  \end{split}
\end{equation}
where
$\mathcal{A}^*_1 \lambda_t(dx_1, dx_2)=\mathcal{A}^*\lambda_t(\cdot,
dx_2)(dx_1)$ and analogously for $\mathcal{A}^*_2$. Note that for
$\mathcal{A}^* =0$ this is completely analogous to
\eqref{eq:gammafinite}.

By a lot of abuse of notation, but parallel with \cite{B:91} and
Equation \eqref{eq:newBM}-\eqref{eq:trueWish}, we can also write
\begin{equation}\label{eq:infiniteWish}
  \begin{split}
    d\lambda_t(dx_1, dx_2)&= \left(\mathcal{A}_1^*\lambda_t(dx_1, dx_2) + \mathcal{A}_2^*\lambda_t(dx_1, dx_2) + n \nu(dx_1)\nu(dx_2) \right)dt\\
    &\quad + \int_0^{\infty} \int_{0}^{\infty} \sqrt{\nu \widehat{\otimes} \nu}(dx_1,dx) dB_t^\top(dy,dx) \sqrt{\lambda_t}(dy,dx_2) \\
    & \quad + \int_0^{\infty} \int_0^{\infty}
    \sqrt{\lambda_t}(dx_1,dx) dB_t(dx,dy) \sqrt{\nu \widehat{\otimes}
      \nu} (dy,dx_2) \, ,
  \end{split}
\end{equation}
where heuristically $B(dx,dy)$ is $d \times d$ matrix of Brownian
fields. We shall not develop a framework where this notation makes
sense, but continue with proving that $ \lambda $ is actually a
generalized Feller process, which should be considered the correct
infinite dimensional version of a Wishart process.

By only a slight abuse of notation, we understand $\mathcal{A}^*$, and
in the sequel also $\mathcal{S}^*$ and other linear operators, as
operators acting on both $\mathbb{S}^{d}$-valued measures as well as
$\mathbb{R}^{d \times n}$-valued or $\mathbb{R}^{n \times d}$-valued
ones as in \eqref{eq:OU}.  The mild formulation of \eqref{eq:lambda1},
denoting the semigroup generated by $\mathcal{A}_1^*+ \mathcal{A}_2^*$
by $\mathcal{S}_t^{*,\widehat{\otimes}}$, then reads as
\begin{align*}
  \lambda_t(dx_1, dx_2)&= \mathcal{S}_t^{*,\widehat{\otimes}}\lambda_0(dx_1, dx_2) + n \int_0^t \mathcal{S}^{*,\widehat{\otimes}}_{t-s}\nu(dx_1) \nu(dx_2) ds\\
                       &\quad + \int_0^t \mathcal{S}^{*,\widehat{\otimes}}_{t-s}(\nu(dx_1)dW^{\top}_s \gamma_s(dx_2) + \gamma_s(dx_1)^{\top} dW_t \nu(dx_2))\\
                       &=\mathcal{S}_t^{*, \widehat{\otimes}}\lambda_0(dx_1, dx_2) + n \int_0^t (\mathcal{S}^{*}_{t-s}\nu(dx_1)) (\mathcal{S}^{*}_{t-s}\nu(dx_2) )ds\\
                       &\quad + \int_0^t (\mathcal{S}^{*}_{t-s}\nu(dx_1))dW^{\top}_s  (\mathcal{S}^{*}_{t-s} \gamma_s(dx_2)) \\
                       &\quad + \int_0^t ( \mathcal{S}^{*}_{t-s}\gamma_s(dx_1))^{\top} dW_s (\mathcal{S}^{*}_{t-s}\nu(dx_2)) \,    ,
\end{align*}
where the second equality follows from property
\eqref{eq:semigrouppropnonsym}.

Let now $\beta$ be a linear operator from $Y^*(F)$ to $ F$ where $F$
stands here for $\mathbb{R}^{n\times d}$, or $\mathbb{S}^d$ with the
property that for a constant matrix $A$ with appropriate matrix
dimensions we have
\begin{align}\label{eq:betanew}
  \beta( A \gamma(\cdot))=A\beta(\gamma(\cdot)),\quad \beta(  \gamma(\cdot)A)=\beta(\gamma(\cdot))A.
\end{align}

By means of $\beta$, define now an operator $\widehat{\beta}$ acting
on $\mathbb{R}^{d \times d}$ valued product measures as follows
\begin{align}\label{eq:betanonsym}
  \widehat{\beta}( \gamma_1^{\top}(\cdot) \gamma_2(\cdot))= \beta(\gamma_1 (\cdot))^{\top}
  \beta(\gamma_2 (\cdot)),
\end{align}
where $\gamma_1$ and $\gamma_2$ are either in
$Y^*(\mathbb{R}^{n\times d})$ or in $Y^*(\mathbb{S}^d)$ (in the latter
case the transpose is not needed). Note that \eqref{eq:betanonsym}
implies that $ \widehat{\beta}( \gamma^{\top}(\cdot) \gamma(\cdot))$
is $\mathbb{S}^d_+$-valued.  Applying $\widehat{\beta}$ to $\lambda$
we find
\begin{align*}
  \widehat{\beta}(\lambda_t)
  &=\widehat{\beta}(\mathcal{S}_t^{*,\widehat{\otimes}}\lambda_0) + n \int_0^t \beta(\mathcal{S}^{*}_{t-s}\nu)\beta(\mathcal{S}^{*}_{t-s}\nu )ds\\
  &\quad + \int_0^t \beta(\mathcal{S}^{*}_{t-s}\nu)dW^{\top}_s \beta (\mathcal{S}^{*}_{t-s} \gamma_s) + \int_0^t \beta( \mathcal{S}^{*}_{t-s}\gamma_s)^{\top} dW_s\beta (\mathcal{S}^{*}_{t-s}\nu).
\end{align*}
Defining as in Equation \eqref{eq:kernel} an $\mathbb{S}^d$-valued
kernel via
\[
  K(t)=\beta(\mathcal{S}_t^*\nu),
\]
we obtain the following generalized $\mathbb{S}^d_+$-valued Volterra
equation
\begin{equation}
  \begin{split}\label{eq:Voltwishart}
    V_t&:=\widehat{\beta}(\lambda_t) =\widehat{\beta}(\mathcal{S}_t^{*,\widehat{\otimes}}\lambda_0) + n \int_0^t K(t-s) K(t-s)ds\\
    &\quad + \int_0^t K(t-s)dW^{\top}_s \beta (\mathcal{S}^{*}_{t-s}
    \gamma_s) + \int_0^t \beta( \mathcal{S}^{*}_{t-s}\gamma_s)^{\top}
    dW_s K(t-s),
  \end{split}
\end{equation}
which we call \emph{Volterra Wishart process} in the following
definition.

\begin{definition}\label{def:VolterraWishart}
  For $\beta$, $\widehat{\beta}$ as given in
  \eqref{eq:betanew}-\eqref{eq:betanonsym} and an
  $\mathbb{S}^d$-valued kernel $K(t)$ defined by
  $K(t) =\beta(\mathcal{S}_t^*\nu)$, we call the process defined in
  \eqref{eq:Voltwishart}, \emph{Volterra Wishart process}.
\end{definition}

\begin{remark}\label{rem:squareOU}
  \begin{enumerate}
  \item Note that $\beta(\gamma_t)$ defines an
    $\mathbb{R}^{n\times d}$-valued Volterra OU process, that is,
    \begin{align}\label{eq:VolterraOU}
      X_t:=\beta(\gamma_t)=\beta(\mathcal{S}^*_t \gamma_0)+ \int_0^t dW_s K(t-s).
    \end{align}
    By the definition of $\widehat{\beta}$, the Volterra Wishart
    process
    \[
      V_t=\widehat{\beta}(\lambda_t)= \beta(\gamma_t (\cdot))^{\top}
      \beta(\gamma_t(\cdot))=X^{\top}_t X_t
    \]
    is thus the matrix square of a Volterra OU process, which
    justifies the terminology.
  \item Note that different lifts of the Volterra OU process given in
    \eqref{eq:VolterraOU} are possible, e.g.~the forward process lift
    $f_t(x):=\mathbb{E}[X_{t+x}|\mathcal{F}_t]$. Then, $f_t(0)=X_t$
    and similarly as in \cite[Section 5.2]{cuctei:18} it can be shown
    that $f$ is an infinite dimensional OU process that solves the
    following SPDE (in the mild sense)
    \[
      df_t(x)=\frac{d}{dx} f_t(x) dt + dW_t K(x), \quad f_0(x) =
      \beta(\mathcal{S}^*_x \gamma_0),
    \]
    on a Hilbert space $H$ of absolutely continuous functions (AC)
    with values in $\mathbb{R}^{n \times d}$, precisely
    $ H=\left\{ f \in AC(\mathbb{R}_+, \mathbb{R}^{n \times d}) \, |
      \, \int_0^{\infty} \|f'(x)\|^2 \alpha(x) dx < \infty \right\} $
    where $ \alpha > 0 $ denotes a weight function (compare
    \cite{F:01}). We can then set $\lambda_t(x,y)=f_t^{\top}(x)f_t(y)$
    and define the same Volterra Wishart process as in
    \eqref{eq:Voltwishart} by
    $V_t:=\lambda_t(0,0)=X_t^{\top}X_t$. 
    By It\^o's formula and variation of constants its dynamics can then equivalently be expressed via
    \begin{equation}
      \begin{split}\label{eq:Wishalt}
        V_t:=\lambda_t(0,0)&= f^{\top}_0(t)f_0(t)+ n\int_0^t K(t-s)K(t-s)ds\\
        &\quad +\int_0^t K(t-s)dW^{\top}_sf_s(t-s) + \int_0^t
        f^{\top}_s(t-s)dW_s K(t-s).
      \end{split}
    \end{equation}
    Comparing \eqref{eq:Wishalt} and \eqref{eq:Voltwishart} yields
    \begin{align}\label{eq:relbeta}
      \beta(\mathcal{S}_x^*\gamma_t)=f_t(x)=\mathbb{E}[X_{t+x}| \mathcal{F}_t], \quad x,t \geq 0.
    \end{align}
  \item In the case when $\beta$ and $\mathcal{S}^*$ are as in Example
    \ref{ex:canonical}, \eqref{eq:Voltwishart} reads as
    \begin{align*}
      \int_{\mathbb{R}^2} \lambda(dx_1, dx_2) &= \int_{\mathbb{R}^2}e^{-(x_1+x_2)t}\lambda_0(dx_1,dx_2) + n \int_0^t K(t-s) K(t-s)ds\\
                                              &\quad + \int_0^t\int_0^{\infty} K(t-s)dW^{\top}_s  e^{-x(t-s)} \gamma_s(dx)\\
                                              &\quad + \int_0^t \int_0^{\infty} e^{-x(t-s)} \gamma^{\top}_s(dx)dW_s K(t-s).
    \end{align*}
    Hence by~\eqref{eq:relbeta},
    $\int_0^{\infty} e^{-x(t-s)}\gamma_s(dx)=\mathbb{E}[X_t |
    \mathcal{F}_s]$. This yields exactly equation
    \eqref{eq:VolterraWishart} considered in the introduction. Note that if $\nu$ and in turn $K$ is chosen as in Remark \ref{rem:frac}, this Volterra Wishart process has exactly the roughness properties desired in rough covariance modeling.

  \end{enumerate}
\end{remark}

In the following remark we list several properties of Volterra Wishart
processes.
\begin{remark}
  \begin{enumerate}
  \item Note that the marginals of $V$ are Wishart distributed as they
    arise from squares of Gaussians.

  \item In order to bring \eqref{eq:lambda1} in a ``standard'' Wishart
    form (with the matrix square root) as in \eqref{eq:Wishart} by
    replacing $\gamma(dx)$ by $\sqrt{\lambda}(dx,dy)$ new notation has
    to be introduced (compare with \eqref{eq:infiniteWish}).
  \item Nevertheless, both the drift and the diffusion characteristic
    of $\lambda$ depend linearly only on $\lambda$, e.g.
    \begin{align*}
      \frac{d[\lambda_{ij}(dx_1,dx_2), \lambda_{kl}(dy_1,dy_2)]_t}{dt}&= (K(x_1) K(y_1))_{ik} \lambda_{t,jl}(dx_2,dy_2)\\
                                                                      &\quad +(K(x_1) K(y_2))_{il} \lambda_{t,jk}(dx_2,dy_1)\\
                                                                      &\quad + (K(x_2)K(y_1))_{jk}\lambda_{t,il}(dx_1,dy_2)\\
                                                                      &\quad +
                                                                        (K(x_2)K(y_2))_{jl}\lambda_{t,ik}(dx_1,dy_1) \, ,
    \end{align*}
    which indicates that $(\lambda_t)_{t\geq 0}$ is Markovian on its
    own. This is shown rigorously below.
  \end{enumerate}
\end{remark}

Using Theorem \ref{th:invariantspace} we now show that $\lambda$ is a
generalized Feller process on $(\widehat{\mathcal{E}},\widehat{\rho})$
with weight function $\widehat{\rho}$ satisfying
\begin{align}\label{eq:weightsquare}
  \widehat{\rho}(\gamma \widehat{\otimes} \gamma) = \rho(\gamma).
\end{align}
We also prove that this generalized Feller process is \emph{affine},
in the sense that its Laplace transform is exponentially affine in the
initial value. The process $\lambda$ can therefore be viewed as an
\emph{infinite dimensional Wishart process} on $\widehat{\mathcal{E}}$
analogously to \cite{B:91, CFMT:11}.

\begin{theorem}
  The process $\lambda$ defined in \eqref{eq:lambda} is Markovian on
  $\widehat{\mathcal{E}}$. The corresponding semigroup is a
  generalized Feller semigroup on
  $ \mathcal{B}^{\widehat{\rho}}(\widehat{\mathcal{E}}) $, where
  $\widehat{\rho}$ satisfies \eqref{eq:weightsquare}. Moreover, for
  $y \in Y(\mathbb{R}^{n\times d})$
  \begin{align}\label{eq:Laplacetrafo}
    \mathbb{E}_{\lambda_0} \left[\exp\left( - \langle y \widehat{\otimes} y  , \lambda_t  \rangle\right)\right]= \exp(-\phi_t- \langle \psi_t, \lambda_0\rangle),
  \end{align}
  where $\psi$ and $\phi$ satisfy the following Riccati differential
  equations, namely $\psi_0=y\widehat{\otimes} y $ and
  $\partial_t \psi_t= R(\psi_t)$ in the mild sense with
  $R: \widehat{\mathcal{E}}_* \to \widehat{\mathcal{E}}_*$ given by
  \begin{align*}
    R(y \widehat{\otimes} y)(x_1,x_2) &= \mathcal{A} y(x_1)
                                        \widehat{\otimes} y (x_2)+ y(x_1) \widehat{\otimes} \mathcal{A} y
                                        (x_2)\\
                                      &\quad - 2 \int_0^{\infty} \int_0^{\infty} y(dx_1) \widehat{\otimes}y(dx) \nu
                                        \widehat{\otimes} \nu (dx,dy) y(dy)\widehat{\otimes}y(dx_2)
  \end{align*}
  and $\phi_0=0$ and $\partial_t \phi_t= F(\psi_t)$ with
  $F: \widehat{\mathcal{E}}_* \to \mathbb{R}$ given by
  \[
    F(y \widehat{\otimes} y) = n \langle y \widehat{\otimes} y, \nu
    \widehat{\otimes} \nu\rangle .
  \]
\end{theorem}

\begin{proof}
  We apply Theorem \ref{th:invariantspace} and Corollary
  \ref{cor:closed_invariant_subspace} with
  \[
    q: \mathcal{Y}^*(\mathbb{R}^{n\times d}) \to
    \widehat{\mathcal{E}}, \, \gamma \mapsto \gamma \widehat{\otimes}
    \gamma= \gamma(\cdot)^{\top}\gamma(\cdot).
  \]
  Observe that this is a continuous map, since we use the dual
  topology $ \sigma(\widehat{\mathcal{E}}, Y \otimes Y) $ on
  $\widehat{\mathcal{E}}$ and the respective polar
  $\widehat{\mathcal{E}}_*$ defined by
  \eqref{eq:polarsquare}. Consider now the following set of Fourier
  basis elements
  \begin{equation*}
    \begin{split}
      \widehat{\mathcal{D}}=\{ f_y: \widehat{\mathcal{E}} \to [0,1];
      \lambda \mapsto \exp( - \langle y \widehat{\otimes} y , \lambda
      \rangle) \, | &\, y \in Y(\mathbb{R}^{n\times d})\}
    \end{split}
  \end{equation*}
  which is dense in $\mathcal{B}^{\widehat{\rho}}(\widehat{\mathcal{E}})$ by
  the very definition of the dual topology.  We check now that the
  generalized Feller semigroup $P^{\text{(OU)}}$ corresponding to
  \eqref{eq:OU} satisfies Assumption \eqref{eq:assumptionP1} for
  $f \in \widehat{\mathcal{D}}$ , i.e. for every
  $f \in \widehat{\mathcal{D}}$ there exists some $g$ such that
  \begin{align}\label{eq:POU}
    P^{\text{(OU)}}_t (f \circ q) = g \circ q \, .
  \end{align}
  Hence we need to compute
  $ \mathbb{E}_{\gamma_0} \left[\exp\left( - \langle y
      \widehat{\otimes} y , \gamma_t \widehat{\otimes}\gamma_t
      \rangle\right)\right].  $ By Lemma \ref{lem:Laplacetrafo} this
  expression is given by \eqref{eq:Laplacetrafoexplicit}. Therefore
  \eqref{eq:POU} is clearly satisfied. This proves the first
  assertion.  Concerning the affine property, we can deduce from Lemma
  \ref{lem:Laplacetrafo} that $\psi$ and $\phi$ are given by
  \begin{align*}
    \psi_t&=(2q_t(y \widehat{\otimes} y )+\operatorname{Id}_d)^{-1}(\mathcal{S}_t y \widehat{\otimes} \mathcal{S}_t y), \\
    \phi_t&=\frac{n}{2}\log\det(2q_t(y \widehat{\otimes} y )+ \operatorname {Id}_d).
  \end{align*}
  with $q_t$ given in  Lemma
  \ref{lem:Laplacetrafo}. 
  Taking derivatives then leads to the form of the Riccati differential
  equations.
\end{proof}

The following lemma provides an explict expression for the Laplace
transform of $\gamma_t \widehat{\otimes}\gamma_t$. This ressembles not
surprisingly the Laplace transfrom of a non-central Wishart
distribution with $n$ degrees of freedom.

\begin{lemma}\label{lem:Laplacetrafo}
  Let $ \gamma $ be an Ornstein-Uhlenbeck process as defined in
  \eqref{eq:OU}. Then for $y \in Y(\mathbb{R}^{n\times d})$, the
  Laplace transform of $\gamma_t \widehat{\otimes}\gamma_t$ is given
  by
  \begin{equation}
    \begin{split}\label{eq:Laplacetrafoexplicit}
      \mathbb{E}_{\gamma_0}\left[\exp(-\langle y \widehat{\otimes} y ,
        \gamma_t \widehat{\otimes}\gamma_t \rangle)\right] &=
      \det(2q_t(y \widehat{\otimes} y )+ \operatorname {Id}_d)^{-\frac{n}{2} }\\
      &\quad \times \exp(-\langle (2q_t(y \widehat{\otimes} y
      )+\operatorname{Id}_d)^{-1}(\mathcal{S}_t y \widehat{\otimes}
      \mathcal{S}_t y) , \gamma_0 \widehat{\otimes} \gamma_0 \rangle),
    \end{split}
  \end{equation}
  where
  $q_t(y \widehat{\otimes} y ) = \int_0^t \int_0^{\infty}
  \int_0^{\infty} \mathcal{S}^*_{s}\nu(dx_1) y^{\top}(x_1)
  y(x_2)\mathcal{S}^*_{s}\nu(dx_2) ds $.

\end{lemma}

\begin{proof}
  Assume for simplicity first that $\mathcal{A}^*$ is equal to
  $0$. Then \eqref{eq:OU} becomes
  \[
    \gamma_t(dx)= \gamma_0(dx) + W_t\nu(dx).
  \]
  Fix $y \in Y(\mathbb{R}^{n\times d})$ such that
  $\int_0^{\infty} y(x) \nu(dx)$ is well defined.  We then have
  \begin{align*}
    \langle y \widehat{\otimes} y  , \gamma_t \widehat{\otimes}\gamma_t  \rangle &=
                                                                                   \langle y \widehat{\otimes} y ,(\gamma_0 + W_t\nu) \widehat{\otimes} (\gamma_0 + W_t\nu)\rangle\\
                                                                                 &= \langle y \widehat{\otimes} y , \gamma_0  \widehat{\otimes} \gamma_0 \rangle + \langle y \widehat{\otimes} y , \gamma_0 \widehat{\otimes}  W_t\nu \rangle +\langle y \widehat{\otimes} y ,   W_t\nu \widehat{\otimes} \gamma_0 \rangle\\
                                                                                 &\quad + \langle y \widehat{\otimes} y , W_t\nu \widehat{\otimes} W_t\nu \rangle.
  \end{align*}
  Note now that
  \begin{align*}
    \langle y \widehat{\otimes} y , \gamma_0 \widehat{\otimes}  W_t\nu \rangle&=
                                                                                \Tr\left(\left(W_t
                                                                                \int_0^{\infty} \int_0^{\infty} \nu(dx_2) y^{\top}(x_1) y(x_2) \gamma_0^{\top}(dx_1)\right) \right)\\
                                                                              &=:\Tr(W_t a),\\
    \langle y \widehat{\otimes} y ,   W_t\nu \widehat{\otimes} \gamma_0 \rangle                                                                              &= \Tr\left(\left(\int_0^{\infty} \int_0^{\infty}\gamma_0(dx_2)  y^{\top}(x_1) y(x_2) \nu(dx_1)\right) W^{\top}_t\right)\\
                                                                              &=:\Tr(a_1 W^{\top}_t)=\Tr( W_t a_1^{\top})=\Tr(W_t a),\\
    \langle y \widehat{\otimes} y , W_t\nu \widehat{\otimes} W_t\nu \rangle&= 
                                                                             \Tr\left(\left(
                                                                             \int_0^{\infty} \int_0^{\infty} \nu(dx_2) y^{\top}(x_1) y(x_2)\nu(dx_1)\right) W_t^{\top}W_t\right)\\
                                                                              &=: \Tr(b W_t^{\top} W_t),
  \end{align*}
  where $a\in \mathbb{R}^{d \times n}$,
  $a_1 \in \mathbb{R}^{n \times d}$, $b \in \mathbb{R}^{d\times d}$
  and $a=a_1^{\top}$.
  
  For the following calculation let $n=1$. Then using these
  expressions, we find
  \begin{align*}
    &\mathbb{E}\left[\exp(-\langle y \widehat{\otimes} y  , \gamma_t \widehat{\otimes}\gamma_t  \rangle)\right]\\
    &\quad =\exp(-\langle y \widehat{\otimes} y , \gamma_0  \widehat{\otimes} \gamma_0 \rangle)\mathbb{E}\left[ \exp( -2 \Tr(W_ta) - \Tr(bW_t^{\top} W_t)\right]\\
    &\quad =\exp(-\langle y \widehat{\otimes} y , \gamma_0  \widehat{\otimes} \gamma_0 \rangle)\frac{1}{(2\pi)^{\frac{d}{2}}  t^{\frac{d}{2}}}
      \int_{\mathbb{R}^{1 \times d}} e^{-2\Tr(xa) - \Tr(bx^{\top}x) -\frac{1}{2t} x x^{\top}}dx\\
    &\quad =\exp(-\langle y \widehat{\otimes} y , \gamma_0  \widehat{\otimes} \gamma_0 \rangle)\\
    &\quad \quad \times \frac{1}{\det(2b+ \frac{1}{t} \operatorname {Id}_d)^\frac{1}{2} t^{\frac{d}{2}}}\frac{1}{(2\pi)^{\frac{d}{2}} }
      \int_{\mathbb{R}^{1 \times d}} e^{-2 xa - \frac{1}{2}x (2b+ \frac{1}{t} \operatorname {Id}_d) x^{\top} }\det(2b+ \frac{1}{t} \operatorname {Id}_d)^{\frac{1}{2}}dx\\
    &\quad =\frac{1}{\det(2b+ \frac{1}{t} \operatorname {Id}_d)^\frac{1}{2} t^{\frac{d}{2}}}\exp(-\langle y \widehat{\otimes} y , \gamma_0  \widehat{\otimes} \gamma_0 \rangle)\exp(2 a^{\top} (2b+ \frac{1}{t} \operatorname {Id}_d)^{-1} a),
  \end{align*}
  where in the last line we used the formula for the moment generating
  function of a Gaussian random variable with covariance
  $(2b+ \frac{1}{t} \operatorname {Id}_d)^{-1}$.  Simplifiying further
  yields
  \begin{align}
    &\mathbb{E}\left[\exp(-\langle y \widehat{\otimes} y  , \gamma_t \widehat{\otimes}\gamma_t  \rangle)\right] \notag\\
    &\quad=\frac{1}{\det(2b+ \frac{1}{t} \operatorname {Id}_d)^\frac{1}{2} t^{\frac{d}{2}}}
      \exp(\langle (2 b(2b+ \frac{1}{t} \operatorname {Id}_d)^{-1}-\operatorname{Id}_d) (y \widehat{\otimes} y), \gamma_0  \widehat{\otimes} \gamma_0 \rangle) \notag \\
    &\quad =\frac{1}{\det(2bt+ \operatorname {Id}_d)^{\frac{1}{2} }}
      \exp(\langle -(\operatorname{Id}_d + 2bt)^{-1} (y \widehat{\otimes} y), \gamma_0  \widehat{\otimes} \gamma_0 \rangle). \label{eq:equationWishart}
  \end{align}
  For general $n$, note that we can write
  \[
    W^{\top}_t W_t= \sum_{j=1}^n W^{\top}_{j,t} W_{j,t},
  \]
  where the $W_j$ are the rows of $W$ and thus take values in
  $\mathbb{R}^{1 \times d }$.  Similary
  \[
    \Tr(W_t a)= \Tr\left(\sum_{j=1}^n W_{j,t} \left( \int_0^{\infty}
        \int_0^{\infty} \nu(dx_2) y^{\top}(x_1) y(x_2)
        \gamma_{0,j}^{\top}(dx_1)\right)\right) =:\sum_{j=1}^n W_{j,t}
    a_j,
  \]
  where $\gamma_{0,j}$ are the rows of $\gamma_0$. Using the
  independence of all $W_j$ and applying \eqref{eq:equationWishart}
  then leads to
  \[
    \mathbb{E}\left[\exp(-\langle y \widehat{\otimes} y , \gamma_t
      \widehat{\otimes}\gamma_t \rangle)\right] = \frac{1}{\det(2bt+
      \operatorname {Id}_d)^{\frac{n}{2} }} \exp(-\langle
    (\operatorname{Id}_d + 2bt)^{-1} (y \widehat{\otimes} y), \gamma_0
    \widehat{\otimes} \gamma_0 \rangle).
  \]
  The general case for $\mathcal{A}^* \neq 0$ can now be traced back
  to this situation.  Indeed, by the variation of constants formula,
  $\gamma_t$ is given by
  \[
    \gamma_t=\mathcal{S}^*_t \gamma_0 + \int_0^t dW_s
    \mathcal{S}^*_{t-s}\nu(dx).
  \]
  Therefore we need to replace $bt$ by
  \[
    q_t=\int_0^t\int_0^{\infty} \int_0^{\infty}
    \mathcal{S}^*_{t-s}\nu(dx_1) y^{\top}(x_1)
    y(x_2)\mathcal{S}^*_{t-s}\nu(dx_2) ds
  \]
  and $\gamma_0$ by $\mathcal{S}^*_t \gamma_0$.  This then yields
  \eqref{eq:Laplacetrafoexplicit}. Note that this now holds for
  general $y \in Y(\mathbb{R}^{n \times d})$ even if
  $\int_0^{\infty} y(x) \nu(dx)$ is not necessarily well defined.
\end{proof}

\section{(Rough) Volterra type affine covariance models}\label{sec:cov}

The goal of this section is to apply the above constructed affine covariance models for multivariate stochastic volatility models with $d$ assets. 
We exemplify this with the Volterra Wishart process of Section \ref{sec:square-OU} and define a (rough) multivariate Volterra Heston type model with possible jumps in the price process. Roughness can be achieved by specifing $\nu$ and in turn the kernel of the Volterra Wishart process as in Remark \ref{rem:frac}. 
The log-price process denoted by $P$ and taking values in $\mathbb{R}^d$  evolves according to 
\begin{equation}
\begin{split}\label{eq:logprice}
dP_t&=-\frac{1}{2}\diag(V_t)dt - \int_{\mathbb{R}^d} (e^{\xi}- \mathbf{1} - \xi) \Tr(V_t m(d\xi)) + X_t^{\top} dB_t\\
&\quad  + \int_{\mathbb{R}^d} \xi (\mu^P(d\xi) - \Tr(V_t m(d\xi)),
\end{split}
\end{equation}
where $X_t$ denotes the Volterra OU process defined in Remark \ref{rem:squareOU}, $\mathbf{1}$ the vector in $\mathbb{R}^d$ with all entries being $1$ and $e^{\xi}$ has to be understood componentwise. Moreover, $B_t$ is  an $\mathbb{R}^n$-valued Brownian motion, which can be correlated with the matrix Brownian motion $W$ appearing in \eqref{eq:OU} as follows
\[
B_t= W_t\rho + \sqrt{(1 -\rho^{\top} \rho)} \widetilde{B}_t.
\]
Here, $\widetilde{B}_t$ is  an $\mathbb{R}^n$-valued Brownian motion independent of $W$ and $\rho \in \mathbb{R}^d$. Moreover, $\mu^P$ denotes the random measure of the jumps with compensator $\Tr(V m(d\xi))$, where $V$ is the Volterra Wishart process of \eqref{eq:Voltwishart} and $m$ a positive semi-definite measure supported on $\mathbb{R}^d$. 

As a corollary of Section 5 and \cite[Section 5]{C:11} we obtain the following result, namely that the log-price process together with the infinite dimensional Wishart process $\lambda$ given in \eqref{eq:lambda}  is an affine Markov process.

Before formulating the precise statement, note that the continuous covariation\footnote{Here, the brackets stand for the covariation and not for the pairing.} $\langle P_{i}, \lambda_{kl}(dx_1, dx_2)\rangle_t $  is given by
\begin{align*}
\frac{\langle P_{i}, \lambda_{kl}(dx_1, dx_2)\rangle_t}{dt}&= (\beta^{\top}(\gamma_t)\gamma_t(dx_1))_{il}(\nu(dx_2)\rho)_k\\
&\quad +(\beta^{\top}(\gamma_t)\gamma_t(dx_1))_{ik}(\nu(dx_2)\rho)_l,
\end{align*}
where $\gamma$ is the infinite dimensional OU-process of \eqref{eq:OU}.
Note that $\beta^{\top}(\gamma_t)\gamma_t(dx_1)$ can also be written as linear map from $\widehat{\mathcal{E}} \to Y^*(\mathbb{S}^d)$ which we denote be $\widetilde{\beta}$, i.e.
\begin{align} \label{eq:betatilde}
\widetilde{\beta}(\lambda_t)(dx_1)=\beta^{\top}(\gamma_t)\gamma_t(dx_1).
\end{align}
In the standard example of \ref{ex:canonical}, we have
$ \widetilde{\beta}(\lambda)(dx_1) = \int_{x_2} \lambda(dx_1, dx_2)$. The adjoint operator of $\widetilde{\beta}$ from $Y(\mathbb{S}^d)$ to $Y(\mathbb{R}^{n\times d}) \widehat{\otimes} Y(\mathbb{R}^{n\times d})$ is denoted by $ \widetilde{\beta}_*$ and given by
\[
\langle \widetilde{\beta}(\lambda), y \rangle= \langle \lambda, \widetilde{\beta}_*(y)\rangle, \quad y \in Y(\mathbb{S}^d),
\]
where the brackets are the pairings in the respective spaces. With this notation we are now ready to state the result. Its proof is a combination of the results of Section~\ref{sec:square-OU} and \cite[Section 5]{C:11}.

\begin{corollary}
The joint process $(\lambda, P)$ with $\lambda$ defined in \eqref{eq:lambda} 
and $P$ defined in \eqref{eq:logprice} 
is Markovian with state space $(\widehat{\mathcal{E}}, \mathbb{R}^d)$. It is affine in the sense that for $(y, v) \in Y(\mathbb{R}^{n \times d}) \times \mathbb{R}^d$, we have
 \begin{align}\label{eq:Laplacetrafojoint}
    \mathbb{E}_{\lambda_0, P_0} \left[\exp\left( - \langle y \widehat{\otimes} y  , \lambda_t  \rangle + \text{\emph{i}} v^{\top} P_t\right) \right]= \exp(-\phi_t- \langle \psi_t, \lambda_0\rangle + \text{\emph{i}} v^{\top} P_0).
  \end{align}
 The function $\psi$ satisfies the following Riccati differential
  equations, namely $\psi_0=y\widehat{\otimes} y $ and
  $\partial_t \psi_t= R(\psi_t,\text{\emph{i}} v )$ in the mild sense with
  $R: \widehat{\mathcal{E}}_* \times \text{\emph{i}} \mathbb{R}^d \to \widehat{\mathcal{E}}_*$ given by
  \begin{align*}
    R(y \widehat{\otimes} y, \text{\emph{i}} v)(x_1,x_2) &= \mathcal{A} y(x_1)
                                        \widehat{\otimes} y (x_2)+ y(x_1) \widehat{\otimes} \mathcal{A} y
                                        (x_2)\\
                                      &\quad - 2 \int_0^{\infty} \int_0^{\infty} y(dx_1) \widehat{\otimes}y(dx) \nu
                                        \widehat{\otimes} \nu (dx,dy) y(dy)\widehat{\otimes}y(dx_2) \\
                                        &\quad+ \frac{1}{2} \sum_{i=1}^d \text{\emph{i}} v_i \widehat{\beta}_* (e_i e_i^{\top})(x_1, x_2)\\ &\quad+ \widehat{\beta}_*(\int_{\mathbb{R}^d}(\text{\emph{i}} v^{\top} (e^{\xi}- \mathbf{1} - \xi) )m(d\xi))(x_1,x_2)\\
                                        &\quad + \frac{1}{2} \widehat{\beta}_*(v v^{\top})(x_1, x_2)\\ &\quad + \widetilde{\beta}_* (\int_0^{\infty} y(\cdot) \widehat{\otimes}y(x)\nu(dx))(x_1,x_2) \rho \text{\emph{i}} v^{\top}\\ &\quad + \text{\emph{i}} v \rho^{\top} \widetilde{\beta}_* (\int_0^{\infty} \nu(dx) y(\cdot) \widehat{\otimes}y(x))(x_1,x_2) \\
                                        &\quad -\widehat{\beta}_*( \int_{\mathbb{R}^d}( \exp(\text{\emph{i}} v^{\top} \xi ) -1 - \text{\emph{i}} v^{\top} \xi  ) m(d\xi))(x_1, x_2),
  \end{align*}
  where $\widehat{\beta}_*$ and $\widetilde{\beta}_*$ are the adjoint operators of $\widehat{\beta}$ given in \eqref{eq:betanonsym} and $\widetilde{\beta}$ given in \eqref{eq:betatilde}, respectively.
  The function $\phi$ satisfies $\phi_0=0$ and $\partial_t \phi_t= F(\psi_t)$ with
  $F: \widehat{\mathcal{E}}_* \to \mathbb{R}$ given by
  \[
    F(y \widehat{\otimes} y) = n \langle y \widehat{\otimes} y, \nu
    \widehat{\otimes} \nu\rangle .
  \]
\end{corollary}

\begin{remark}
In a similar spirit one can define multivariate affine covariance models with  the affine Volterra jump process $V$ given in \eqref{eq:Voltrep}. 
The log-price process (under some risk neutral measure) evolves then according to 
\begin{align*}
dP_t&=-\frac{1}{2}\diag(V_t)dt - \int_{\mathbb{R}^d} (e^{\xi}- \mathbf{1} - \xi) \Tr(V_t m(d\xi)) + \sqrt{V}_t dB_t \\
&\quad + \int_{\mathbb{R}^d} \xi (\mu^P(d\xi) - \Tr(V_t m(d\xi)),
\end{align*}
where $B$ is a $d$-dimensional Brownian motion and  the jump measure $m$ of $P$ and $\mu$ of the Markovian lift $\lambda$ as given in \eqref{eq:SPDEmain} can be the marginals of some common measure supported on $\mathbb{S}^d_+ \times \mathbb{R}^d $.
\end{remark}

\end{document}